\documentclass{amsart}
\usepackage{amsmath, amssymb}
\usepackage{enumitem}
\usepackage{float}
\usepackage{nicematrix}
\usepackage{booktabs}
\usepackage[all]{xy}
\usepackage{makecell}
\usepackage{fancyvrb}

\usepackage{hyperref}

\theoremstyle{plain}
\newtheorem{theorem}{Theorem}[section]
\newtheorem{proposition}[theorem]{Proposition} 

\newtheorem{lemma}[theorem]{Lemma}

\theoremstyle{definition}

\newtheorem{remark}[theorem]{Remark}
\newtheorem{claim}[theorem]{Claim}

\numberwithin{equation}{section}

\newenvironment{claimproof}
	{\proof}
	{\endproof}

\newcommand{\Aut}{\operatorname{Aut}}
\newcommand{\Fix}{\operatorname{Fix}}
\newcommand{\Gal}{\operatorname{Gal}}
\newcommand{\gon}{\operatorname{gon}}
\newcommand{\ord}{\operatorname{ord}}
\newcommand{\rank}{\operatorname{rank}}
\newcommand{\Ker}{\operatorname{Ker}}
\newcommand{\img}{\operatorname{Im}}
\newcommand{\id}{\operatorname{id}}

\begin{document}

%%%%%%%%%%%%%%%%%%%%%%%%%%%%%%%%%%%%%%%%%%%

\title[Galois lines for a canonical curve of genus 4, III]{Galois lines for a canonical curve of genus 4, III: non-cyclic Galois lines}

\author[S.~Kato]{Shotaro Kato}
\address{Graduate School of Science and Technology, Niigata University, 8050 Ikarashi-ninocho Nishi-ku, Niigata 950-2181, Japan}
\email{shotarokato0308@gmail.com}
\curraddr{Simplex Holdings, Inc., Toranomon Hills Mori Tower 19F, 1-23-1 Toranomon, Minato-ku, Tokyo 105-6319, Japan}
\email{kato.sh@simplex.ne.jp}

\author[J.~Komeda]{Jiryo Komeda}
\address{Department of Mathematics, Kanagawa Institute of Technology, 1030 Shimo-Ogino, Atsugi 243-0292, Japan}
\email{komeda@gen.kanagawa-it.ac.jp}

\author[T.~Takahashi]{Takeshi Takahashi}
\address{Education Center for Engineering and Technology, Faculty of Engineering, Niigata University, 8050 Ikarashi-ninocho Nishi-ku, Niigata 950-2181, Japan}
\email{takeshi@eng.niigata-u.ac.jp}

\thanks{
	This work was supported by JSPS KAKENHI Grant Numbers JP18K03228, JP19K03441 and 	JP25K06930. %\\
	%\indent Preprint. Ver. November 17, 2026.
}

%%%%%%%%%%%%%%%%%%%%%%%%%%%%%%%%%%%
\begin{abstract}
	Let $C \subset \mathbb{P}^3$ be a canonical curve of genus $4$ over an algebraically closed field $k$ of characteristic zero. For a line $l \subset \mathbb{P}^3$, we consider the projection $\pi_l: C \to \mathbb{P}^1$ from $l$ and the induced extension of function fields $\pi_l^*: k(\mathbb{P}^1)\hookrightarrow k(C)$. A line $l$ is called an \emph{$S_3$-line} (resp.\ a \emph{$K_4$-line}) if the extension $k(C)/\pi_l^*(k(\mathbb{P}^1))$ is Galois and its Galois group is isomorphic to the symmetric group $S_3$ on three letters (resp.\ the Klein four-group $K_4$). We prove that the number of $S_3$-lines (resp.\ $K_4$-lines) is at most $10$ (resp.\ $15$).
\end{abstract}

\maketitle

%%%%%%%%%%%%%%%%%%%%%%%%%%%%%%%%%%%%%%%%%%%%%%%%%
\section{Introduction and Theorem}\label{section:introduction}

The notion of a Galois point, introduced by Yoshihara, has led to extensive research on projections of algebraic varieties and the associated Galois extensions of function fields. Its higher-codimensional analogues include Galois lines, Galois subspaces, and Galois embeddings. The subject has been studied for various classes of curves and hypersurfaces, both in characteristic zero and in positive characteristic, and has generated a substantial literature. We refer the reader to the list of open questions maintained by Fukasawa and Yoshihara \cite{FukasawaYoshiharaProblems} for an overview of the subject and further references. 

Recently, Auffarth and Rahausen investigated Galois subspaces of smooth projective curves, while Auffarth subsequently extended this viewpoint to projective varieties \cite{AuffarthRahausen2021,Auffarth2026}. Their results suggest that a detailed understanding of the automorphism group may lead to an explicit description of Galois subspaces. 

For projective curves, Galois subspaces are linear subspaces of codimension two. In particular, for curves in $\mathbb P^3$, Galois subspaces are precisely Galois lines. For canonical curves of genus~4, the first author \cite{Kato2024,KatoMaster2026} completely determined the numbers of $S_3$-lines and $K_4$-lines using the classification of automorphism groups due to Kuribayashi and Kuribayashi together with computer calculations. Although obtained independently, this approach is closely related in spirit to that of Auffarth and Rahausen, in that both make essential use of the structure of the automorphism group in studying Galois subspaces. 

The present paper approaches the problem from a different direction. By exploiting geometric properties of canonical curves of genus~4 and group-theoretic arguments, we obtain the desired results without relying on classification tables of automorphism groups or computer computations.

Let $k$ be an algebraically closed field of characteristic~$0$, and let $C \subset \mathbb{P}^3$ be a canonical curve of genus~$4$ over $k$, and let $l \subset \mathbb{P}^3$ be a line.

In a series of papers, we have investigated Galois lines for such curves. In~\cite{KomedaTakahashi2021}, we studied non-skew cyclic Galois lines, namely cyclic Galois lines $l$ such that $C \cap l \neq \emptyset$. We proved that the number of $C_4$-lines and $C_5$-lines is at most one. We also explicitly described the curve in the special case where $C$ admits two cyclic trigonal morphisms. In~\cite{KomedaTakahashi2022}, we continued this study by investigating skew cyclic Galois lines, that is, cyclic Galois lines $l$ satisfying $C \cap l = \emptyset$. In this situation, a skew cyclic line is necessarily a $C_6$-line. We proved that the number of such lines is equal to $0,1,3$, or $9$, and completely determined the canonical curves of genus~$4$ which admit nine $C_6$-lines. In the present paper, we continue our investigation of Galois lines for canonical curves of genus~4. 

It is well known that $C$ is a complete intersection of type $(2,3)$ in $\mathbb{P}^3$. For a line $l$, we consider the projection $\pi_l : C \to \mathbb{P}^1$ from $l$, together with the induced extension of function fields $\pi_l^* : k(\mathbb{P}^1) \hookrightarrow k(C)$. Since $\deg C = 6$ and $C$ is not hyperelliptic, the degree of $\pi_l$ satisfies 
\[3 \le \deg \pi_l \le 6.\] 

In~\cite{KomedaTakahashi2021,KomedaTakahashi2022}, we investigated all cyclic Galois lines for canonical curves of genus~4. Since the Galois group associated with a Galois line has order $\deg\pi_l$ and $3\le \deg\pi_l\le 6$, the only possible non-cyclic Galois groups are $S_3$ and $K_4$. The purpose of the present paper is to investigate these two cases.

A line $l$ is called a \emph{Galois line} if the extension $k(C)/\pi_l^*(k(\mathbb{P}^1))$ is Galois. Moreover, a Galois line $l$ is called an \emph{$S_3$-line} (resp.\ a \emph{$K_4$-line}) if its Galois group is isomorphic to the symmetric group $S_3$ on three letters (resp.\ the Klein four-group $K_4$). 
If $l$ is a Galois line, we denote by $G_l$ the group
\[
	G_l := \{\sigma \in \Aut(C) \mid \pi_l \circ \sigma = \pi_l\},
\]
which is naturally isomorphic to the Galois group $\Gal(k(C)/\pi_l^*(k(\mathbb{P}^1)))$. We note that, for each $\sigma \in \Aut(C)$, there exists a unique projective transformation of $\mathbb{P}^3$ whose restriction to $C$ coincides with $\sigma$, since the inclusion $C \subset \mathbb{P}^3$ is given by the canonical embedding.

The main theorem of the present paper is the following. 
\begin{theorem}\label{theorem:main}
	Let $C \subset \mathbb{P}^3$ be a canonical curve of genus $4$ over an algebraically closed field of characteristic $0$. Then the number of $S_3$-lines is at most $10$, and the number of $K_4$-lines is at most $15$. 
\end{theorem}

The first author~\cite{Kato2024, KatoMaster2026} completely determined the numbers of $S_3$-lines and $K_4$-lines by explicit computations. That work is based on the classification of automorphism groups of compact Riemann surfaces of genus~$4$ due to I.~Kuribayashi and A.~Kuribayashi \cite{KuribayashiKuribayashi1986, KuribayashiKuribayashi1990}, together with computer calculations using GAP~\cite{GAP4}. In the present paper, we give a proof of Theorem~\ref{theorem:main} which does not rely on these classification tables nor on computer computations.

This paper is organized as follows. In Section~\ref{section:preliminaries}, we recall basic facts on canonical curves of genus~$4$ and Galois lines that will be used throughout the paper. Section~\ref{section:proof} is devoted to the proof of Theorem~\ref{theorem:main}. In Section~\ref{section:examples}, we present an explicit example illustrating the sharpness of the bounds in Theorem~\ref{theorem:main}. Appendix~\ref{appendix:gap} contains the GAP computation used to determine the $S_3$-lines and $K_4$-lines appearing in this example.

%%%%%%%%%%%%%%%%%%%%%%%%%%%%%%%%%%%%%%%%%%%%%%%
\section{Preliminaries} \label{section:preliminaries}

Let $(X:Y:Z:W)$ be homogeneous coordinates on $\mathbb{P}^3$, and let $C \subset \mathbb{P}^3$ be the canonical curve of genus $4$ introduced in Section~\ref{section:introduction}. Two pencils $\Phi, \Phi': C \to \mathbb{P}^1$ are said to be equivalent if there exists a projective transformation $T \in \Aut(\mathbb{P}^1)$ such that $\Phi' = T \circ \Phi$. The following are well-known facts. 

\begin{proposition}[p.~118 of \cite{ACGH1985}, p.~298 of \cite{GriffithsHarris1994}] \label{prop:well-known-facts}
	The curve $C$ is a $(2,3)$-complete intersection; that is, the homogeneous ideal $I(C) \subset k[X,Y,Z,W]$ of $C$ is generated by a quadratic form $Q$ and a cubic form $F$. The degree of $C$ is $6$. The surface $\{Q=0\}$ is a unique quadric surface that contains $C$. The gonality $\gon(C)$ of $C$ is equal to $3$. 
	If $\rank Q = 3$, then $C$ has a unique trigonal morphism $C \to \mathbb{P}^1$ (up to equivalence), which is given by the projection from the vertex of the quadric surface $\{Q=0\}$. 
	If $\rank Q = 4$, then $C$ has exactly two trigonal morphisms $C \to \mathbb{P}^1$ (up to equivalence), which are given by the two projections of the product $\{Q=0\} \cong \mathbb{P}^1 \times \mathbb{P}^1$.
\end{proposition}

Let $l \subset \mathbb{P}^3$ be a line and $\pi_l:C \rightarrow \mathbb{P}^1$ the projection from $l$.  

\begin{proposition}[\cite{KomedaTakahashi2021}] \label{prop:pi_l-determines-l} 
	Let $l,l' \subset \mathbb{P}^3$ be lines. Assume $\deg \pi_l \geq 4$. Then, $\pi_l=\pi_{l'}$ (up to equivalence) if and only if $l=l'$. 
\end{proposition}

\begin{lemma} \label{lemma:representation-when-l:X=Y=0}
	Let $l$ be a Galois line defined by $X=Y=0$. Assume that $\deg \pi_l \geq 4$. Then every element $\sigma \in G_l$ is represented by a matrix of the following form 
	\begin{equation} \label{eq:matrix-when-l:X=Y=0}
		\begin{pmatrix}
			1 & 0 & 0 & 0 \\
			0 & 1 & 0 & 0 \\
			* & * & * & * \\
			* & * & * & *
		\end{pmatrix}.
	\end{equation}
\end{lemma}

\begin{proof}
	We may assume that the projection $\pi_l$ is given by
	\[
		\pi_l \colon (X:Y:Z:W) \longmapsto (X:Y).
	\]
	Let $(a_{ij})$ be a matrix representing $\sigma \in \Aut(C)$. 
	Namely, $\sigma$ acts on $\mathbb{P}^3$ as
	\[
		(X:Y:Z:W) \mapsto (X':Y':Z':W'), %\quad 
		{}^t(X':Y':Z':W') = (a_{ij})\, {}^t(X:Y:Z:W).
	\]
	Since $\pi_l \circ \sigma = \pi_l$, we have $(X:Y) = (X':Y')$ for every point $(X:Y:Z:W) \in C$. Hence,
	\[
		XY' - YX'
		= X(a_{21}X + a_{22}Y + a_{23}Z + a_{24}W) - Y(a_{11}X + a_{12}Y + a_{13}Z + a_{14}W) = 0
	\]
	holds on $C$.
	Since $I(C) = (Q, F)$, where $Q$ is a quadratic form and $F$ is a cubic form, it follows that
	\[
		X(a_{21}X + a_{22}Y + a_{23}Z + a_{24}W)- Y(a_{11}X + a_{12}Y + a_{13}Z + a_{14}W) = \mu Q
	\]
	for some $\mu \in k$. If $\mu\neq 0$, then $Q$ is linear in $Z$ and $W$. Hence for a general $P\in\mathbb P^1$, the fiber $\pi_l^{-1}(P)$, defined by $Q=F=0$, has at most $3$ points, contradicting $\deg\pi_l\ge 4$. Therefore, $\mu = 0$, and the above polynomial vanishes identically. Comparing coefficients, we obtain $a_{11} = a_{22}$ and $a_{ij} = 0$ for $i \neq j$ with $i \in \{1,2\}$ and $j \in \{1,2,3,4\}$. This shows that $\sigma$ is represented by a matrix of the stated form.
\end{proof}

On $S_3$-lines, we have the following: 

\begin{proposition}[The proof of Theorem 4.5 in \cite{Yoshihara2006}, \cite{KomedaTakahashi2022}] \label{prop:generator-of-Gal-S_3-line}
	Let $l$ be an $S_3$-line for $C$. Then, by taking a suitable projective transformation, we have that $l:X=Y=0$, and $G_l$ is generated by the following two elements: 
	\begin{equation} \label{eq:representation-of-S_3}
		\sigma:=\begin{pmatrix}
			1 & 0 & 0 & 0 \\
			0 & 1 & 0 & 0 \\
			0 & 0 & \omega & 0 \\
			0 & 0 & 0 & \omega^2 \\
			\end{pmatrix} \text{ and } 
		\tau:=\begin{pmatrix}
			1 & 0 & 0 & 0 \\
			0 & 1 & 0 & 0 \\
			0 & 0 & 0 & 1 \\
			0 & 0 & 1 & 0 \\
			\end{pmatrix}
	,
	\end{equation}
	where $\omega$ is a primitive cubic root of unity. 
\end{proposition}

For $\sigma$ and $\tau$ in Proposition~\ref{prop:generator-of-Gal-S_3-line}, the fixed locus $\Fix(\sigma)$ consists of the line $Z=W=0$ and the two points $(0:0:1:0)$ and $(0:0:0:1)$, and the $S_3$-line $l$ passes through these two points. The fixed locus $\Fix(\tau)$ consists of the hyperplane $Z-W=0$ and the point $(0:0:-1:1)$, and the line $l$ passes through this point. Hence, we have the following:
	
\begin{proposition} \label{prop:distinct_l_induce_distinct_G_l}
	Let $l$ and $l'$ be $S_3$-lines for $C$, and let $\sigma$ and $\sigma'$ be elements of $G_l$ and $G_{l'}$, respectively, with $\operatorname{ord}(\sigma)=\operatorname{ord}(\sigma')=3$. If $l \neq l'$, then $\langle \sigma \rangle \neq \langle \sigma' \rangle$; in particular, $G_l \neq G_{l'}$. 
\end{proposition}

On $K_4$-lines, we have the following: 

\begin{proposition} 
	\label{prop:generator-of-Gal-K4-line}
	Let $l$ be a $K_4$-line for $C$. Then, by taking a suitable projective transformation, $l:X=Y=0$, and all elements of $G_l$ are 
	\begin{equation} \label{eq:standard-representation-K4}
		I, 
		\begin{pmatrix}
			1 & 0 & 0 & 0 \\
			0 & 1 & 0 & 0 \\
			0 & 0 & -1 & 0 \\
			0 & 0 & 0 & 1 \\
		\end{pmatrix},  
		\begin{pmatrix}
			1 & 0 & 0 & 0 \\
			0 & 1 & 0 & 0 \\
			0 & 0 & 1 & 0 \\
			0 & 0 & 0 & -1 \\
		\end{pmatrix}, \text{ and }
		\begin{pmatrix}
			1 & 0 & 0 & 0 \\
			0 & 1 & 0 & 0 \\
			0 & 0 & -1 & 0 \\
			0 & 0 & 0 & -1 \\
		\end{pmatrix}. 
	\end{equation}
\end{proposition}

\begin{proof}	
	By taking a suitable projective transformation, we may assume that $l:X=Y=0$ and $\pi_l:(X:Y:Z:W) \mapsto (X:Y)$. 	
	Let $\tau_1$ and $\tau_2$ be generators of $G_l \cong K_4$. Namely, $G_l = \{e, \tau_1, \tau_2, \tau_1 \tau_2\}$, where $e$ is the unit, $\tau_1^2 = \tau_2^2 = e$, and $\tau_1 \tau_2 = \tau_2 \tau_1$. Since $\pi_l \circ \tau_i = \pi_l$, by Lemma~\ref{lemma:representation-when-l:X=Y=0}, we have that $\tau_i \in PGL(4,k)$ ($i=1,2$) is represented by a matrix~\eqref{eq:matrix-when-l:X=Y=0}.
	Let $A_i$ and $B_i$ ($i=1,2$) be the matrices of size $2 \times 2$ such that  
	\[	\tau_i = \begin{pmatrix}
			I & O \\
			A_i & B_i \\
		\end{pmatrix},  \]
	where $I$ and $O$ are the identity matrix and the zero matrix of size $2 \times 2$, respectively. Since $\tau_i^2 =e$ and $\tau_1 \tau_2 = \tau_2 \tau_1$, we have that $A_i + B_i A_i = O$, $B_i^2 = I$, $A_1 + B_1 A_2 = A_2 + B_2 A_1$, and $B_1 B_2 = B_2 B_1$. Since $B_i^2 =I$ and $B_1 B_2 = B_2 B_1$, the matrices $B_1$ and $B_2$ are simultaneously diagonalizable. Let 
	\[
		Q^{-1} B_i Q = \begin{pmatrix}	 a_i & 0 \\ 0 & b_i \end{pmatrix} \, (i=1,2)
	\] 
	be the diagonalizations, where $Q$ is an invertible  matrix of size $2 \times 2$, and $a_i, b_i \in k$. We set 
	\[
		P = \begin{pmatrix}	I & O \\ O & Q	\end{pmatrix}. 
	\]
	Then, $P$ is an invertible matrix of size $4 \times 4$ and
	\begin{equation*}
		P^{-1} \tau_i P =
		\begin{pNiceMatrix}
			\Block{1-1}{I} & \Block{1-2}{O} \\
			\Block{2-1}{Q^{-1} A_i} & a_i & 0 \\
             & 0   & b_i
		\end{pNiceMatrix}.
	\end{equation*}
	Let $C_i$ denote $Q^{-1}A_i$. Since $\tau_i^2 = e$ and $\tau_1 \tau_2 = \tau_2 \tau_1$, we have that $a_i^2 = 1$, $b_i^2=1$, 
	\begin{equation} \label{eq:C_i}
		C_i + 
		\begin{pmatrix} 
			a_i & 0  \\
			0 & b_i  
		\end{pmatrix} 
		C_i =O
	\end{equation} 
	and 
	\begin{equation} \label{eq:C_3}
		C_1 + 
		\begin{pmatrix} 
			a_1 & 0  \\
			0 & b_1  
		\end{pmatrix} 
		C_2 = C_2 + 
		\begin{pmatrix} 
			a_2 & 0  \\
			0 & b_2  \\
		\end{pmatrix} 
		C_1.
	\end{equation}
	Let $C_3$ denote the matrix of the left-hand side (or the right-hand side) of Equation~\eqref{eq:C_3}. Then, 
	\[
		P^{-1}\tau_1 \tau_2 P = 
		\begin{pNiceMatrix}
			\Block{1-1}{I} & \Block{1-2}{O} \\
			\Block{2-1}{C_3} & a_1a_2 & 0 \\
                       & 0   & b_1b_2
		\end{pNiceMatrix}.
	\]	
	Since $P^{-1}\tau_1P$, $P^{-1}\tau_2P$, and $P^{-1}\tau_1\tau_2P$ are mutually distinct and none of them is the identity matrix, we see from \eqref{eq:C_i} and \eqref{eq:C_3} that the matrices
	\[
		\begin{pmatrix}
			a_1 & 0 \\
			0 & b_1
		\end{pmatrix}, 
		\begin{pmatrix}
			a_2 & 0 \\
			0 & b_2
		\end{pmatrix}, \text{and}
		\begin{pmatrix}
			a_1a_2 & 0 \\
			0 & b_1b_2
		\end{pmatrix}
	\]
	are mutually distinct and none of them is the identity matrix. Since $a_i^2 = 1$ and $b_i^2=1$, we may assume that 
   \[
		\begin{pmatrix}
			a_1 & 0 \\
			0 & b_1
		\end{pmatrix} = 
		\begin{pmatrix}
			1 & 0 \\
			0 & -1
		\end{pmatrix}, 
		\begin{pmatrix}
			a_2 & 0 \\
			0 & b_2
		\end{pmatrix} = 
		\begin{pmatrix}
			-1 & 0 \\
			0 & 1
		\end{pmatrix}, \text{and}
		\begin{pmatrix}
			a_1a_2 & 0 \\
			0 & b_1b_2
		\end{pmatrix} = 
		\begin{pmatrix}
			-1 & 0 \\
			0 & -1
		\end{pmatrix}. 
   \]
	By Equation~\eqref{eq:C_i}, 
	$C_1 = \begin{pmatrix}
		0 & 0 \\
		* & *
	\end{pmatrix}$ and 
	$C_2 = \begin{pmatrix}
		* & * \\
		0 & 0
	\end{pmatrix}$. Let us denote 
	\[
		C_1 = \begin{pmatrix}
		0 & 0 \\
		c & d
	\end{pmatrix} \text{ and }
	C_2 = \begin{pmatrix}
		e & f \\
		0 & 0
	\end{pmatrix}.
	\]
	Set 
	\[	
		R :=	\begin{pmatrix}
				1 & 0 & 0 & 0 \\
				0 & 1 & 0 & 0 \\
				e/2 & f/2 & 1 & 0 \\
				c/2 & d/2 & 0 & 1
			\end{pmatrix}. 
	\]
	Then, we have that 
	\[
		(PR)^{-1} \tau_1 (PR) = 	\begin{pmatrix}
				1 & 0 & 0 & 0 \\
				0 & 1 & 0 & 0 \\
				0 & 0 & -1 & 0 \\
				0 & 0 & 0 & 1
			\end{pmatrix}, \text{ and }
		(PR)^{-1} \tau_2 (PR) = 	\begin{pmatrix}
				1 & 0 & 0 & 0 \\
				0 & 1 & 0 & 0 \\
				0 & 0 & 1 & 0 \\
				0 & 0 & 0 & -1
			\end{pmatrix}. 
	\]
	Note that the defining equations $X=Y=0$ of the line $l$ remain unchanged under the projective transformations $P$ and $R$. 	
\end{proof}

For a $K_4$-line $l$ for $C$, there exists an element of $G_l$ whose trace is equal to $0$. 
\begin{lemma} \label{lemma:fixed-points-of-trace-0-involution}
	Let $l$ be a $K_4$-line and $\sigma \in G_l$ be an element with trace $0$. The set $\Fix (\sigma) := \{ P \in \mathbb{P}^3 \mid \sigma(P) =P \}$ is the union of two lines: one is the $K_4$-line $l$, and the other is one that does not intersect the curve $C$. In particular, $\sigma$ has exactly two fixed points on $C$.
\end{lemma}
\begin{proof}
	By taking a suitable projective transformation, we may assume that the $K_4$-line $l$ is given by $X = Y = 0$, that the group $G_l$ is represented by the matrices given in \eqref{eq:standard-representation-K4}, and that
	\[
		\sigma =
			\begin{pmatrix}
				1 & 0 & 0 & 0 \\
				0 & 1 & 0 & 0 \\
				0 & 0 & -1 & 0 \\
				0 & 0 & 0 & -1
			\end{pmatrix}.
	\]
	Let $\tau_1$ and $\tau_2$ be the involutions in $G_l$ distinct from $\sigma$.

	Since $\pi_l : C \to \mathbb{P}^1$ is a non-cyclic Galois morphism of degree $4$, no ramification point of $\pi_l$ is totally ramified. Applying the Riemann--Hurwitz formula, we see that $\pi_l$ has exactly $14$ ramification points and that all ramification indices are equal to $2$. Moreover, each ramification point is a fixed point of exactly one of
$\tau_1$, $\tau_2$, and $\sigma$.

	Applying the Riemann--Hurwitz formula again, the genera of the quotient curves $C/\langle \tau_1 \rangle$, $C/\langle \tau_2 \rangle$, and $C/\langle \sigma \rangle$ are equal to $1$ or $2$. (Genus $0$ is excluded, since the quotient morphism $C \to C/\langle \sigma \rangle$ has degree $2$ and $C$ is non-hyperelliptic.) If the genus of $C/\langle \sigma \rangle$ is equal to $1$ (resp.\ $2$), then the involution $\sigma$ has exactly $6$ (resp.\ $2$) fixed points on $C$.

	From the explicit form of the matrix representing $\sigma$, we see that the fixed locus $\Fix(\sigma)$ in $\mathbb{P}^3$ is the union of the two lines $l : X=Y=0$ and $l' : Z=W=0$. Since $l$ is a $K_4$-line, we have $\deg \pi_l = 4$, and therefore $\# (C \cap l) \le 2$. On the other hand, since $C$ is not hyperelliptic, we have $\#(C \cap l') \le 3$. It follows that $\sigma$ cannot have $6$ fixed points on $C$, and hence $\sigma$ has exactly $2$ fixed points and the genus of $C/\langle \sigma \rangle$ is equal to $2$.

	Finally, assume that $\#(C \cap l) = 1$, and let $P$ be the unique point in $C \cap l$. Since $\tau_1$, $\tau_2$, and $\sigma$ are automorphisms of $C$ preserving the line $l$, the point $P$ would be a totally ramified point of $\pi_l$, which contradicts the fact that $\pi_l$ is a non-cyclic Galois morphism. Therefore, we have $\#(C \cap l) = 2$ and $C \cap l' = \emptyset$.
\end{proof}

By Lemma~\ref{lemma:fixed-points-of-trace-0-involution}, we obtain the following. 	
\begin{proposition} \label{prop:two-G_ls-belonging-to-K_4-lines}
	Let $l$ and $l'$ be $K_4$-lines for $C$. Let $\sigma$ and $\sigma'$ be elements in $G_l$ and $G_{l'}$ with traces zero, respectively. If $l \ne l'$ then $\sigma \ne \sigma'$, in particular, $G_l \ne G_{l'}$. 
\end{proposition}

\begin{remark} \label{rem:obtain_l_from_G_l=K4}
	By Proposition~\ref{prop:generator-of-Gal-K4-line}, we have the following. Let $l$ be a $K_4$-line and let $\tau_1, \tau_2 \in G_l$ be distinct elements of order~$2$ with nonzero trace. Each of the fixed loci $\Fix(\tau_1)$ and $\Fix(\tau_2)$ is the union of a plane and a point $P_i$ ($i=1,2$). The line passing through the two points $P_1$ and $P_2$ is precisely the $K_4$-line $l$.
\end{remark}

%%%%%%%%%%%%%%%%%%%%%%%%%%%%%%%%%%%%%%%%%%%%%%%%%%%%
\section{Proof of Theorem~\ref{theorem:main}}\label{section:proof} 

In this section, we study the numbers of $S_3$-lines and $K_4$-lines by dividing the argument into several cases. By combining the results obtained in each case, we prove Theorem~\ref{theorem:main}. We note that some intermediate bounds obtained in this section are not sharp, although Theorem~\ref{theorem:main} itself gives sharp upper bounds.

The following lemma is immediate from the definition of equivalence of pencils.

\begin{lemma}\label{lemma:action-on-trigonal}
	Let $f:C\to\mathbb P^1$ be a pencil.
	If $\sigma\in\Aut(C)$ satisfies $f\circ\sigma\simeq f$, then there exists a unique $T\in\Aut(\mathbb P^1)$ such that $f\circ\sigma = T\circ f$. 
	\[
		\xymatrix{
  			C \ar[r]^{\sigma} \ar[d]_{f}^{\hspace{3.3ex} \circlearrowleft} & C \ar[d]^{f} \\
  			\mathbb{P}^1 \ar[r]^{T} & \mathbb{P}^1
		}
	\]
	Hence any subgroup $G\subset\Aut(C)$ preserving $f$ induces a homomorphism
	\begin{equation}\label{eq:fundamental-action}
		\varphi:G\longrightarrow\Aut(\mathbb P^1) \qquad \sigma \longmapsto T,
	\end{equation}
	and we have a short exact sequence
	\begin{equation}\label{eq:fundamental-exact-seq}
		1 \longrightarrow \Ker\varphi \longrightarrow G \xrightarrow{\ \varphi\ } \img\varphi \longrightarrow 1.
	\end{equation}
\end{lemma}

Our proof proceeds as follows. To the set of Galois lines (either $S_3$-lines or $K_4$-lines), we associate a suitable finite subgroup $G$ of $\Aut(C)$. For a trigonal morphism $f=g^1_3$ or $h^1_3$, the induced action~\eqref{eq:fundamental-action} of $G$ gives rise to the short exact sequence~\eqref{eq:fundamental-exact-seq} with image contained in $\Aut(\mathbb P^1)$. Using the classification of finite subgroups of $\Aut(\mathbb P^1)$, we obtain strong restrictions on $G$, which in turn yield an upper bound for the number of Galois lines.

%%%%%
\subsection{The number of $S_3$-lines when $\rank Q=3$}

Assume that $\rank Q=3$, so that $C$ admits a unique trigonal pencil $g^1_3$. Let $G$ be the subgroup of $\Aut(C)$ generated by the groups $G_l$, where $l$ ranges over all $S_3$-lines. Since every element of $G$ preserves the pencil $g^1_3$ up to equivalence, Lemma~\ref{lemma:action-on-trigonal} yields a homomorphism $\varphi:G\to\Aut(\mathbb P^1)$ and the associated short exact sequence
\begin{equation}\label{eq:exact-seq-one-trigonal}
	1 \longrightarrow \Ker\varphi \longrightarrow G
	\xrightarrow{\ \varphi\ } \img\varphi\longrightarrow 1.
\end{equation}

\begin{proposition}
	Assume that $\rank Q = 3$. Then:  
		\begin{enumerate}[label=\upshape(\Roman*)]
			\item if $\Ker \varphi =1$, then the number of $S_3$-lines is at most $10$;
			\item if $\Ker \varphi \ne 1$, then the number of $S_3$-lines is at most $4$.
		\end{enumerate}	
\end{proposition}
\begin{proof}

	(I) Assume that $\Ker \varphi =1$. Then, $G \cong \img \varphi \subset \Aut(\mathbb{P}^1)$.  
	It is well known that any finite subgroup of $\Aut(\mathbb{P}^1)$ is isomorphic to one of the following groups: a cyclic group $C_m$ of order $m$, a dihedral group $D_m$ of order $2m$, the alternating group $A_4$ on $4$ letters, the symmetric group $S_4$ on $4$ letters, or the alternating group $A_5$ on $5$ letters. 
	The number of $C_3$ subgroups contained in $S_3$ subgroups of $C_m$ (resp.\ $D_m$, $A_4$, $S_4$, $A_5$) is equal to $0$ (resp.\ at most $1$, equal to $0$, equal to $4$, equal to $10$). Using Proposition~\ref{prop:distinct_l_induce_distinct_G_l}, we conclude that the number of $S_3$-lines is at most $10$. 

	(II) Assume that $\Ker \varphi \ne 1$. We first note the following. 
	\begin{claim}
		The trigonal morphism $g^1_3$ is Galois and $\Ker\varphi \cong C_3$.
	\end{claim}
	\begin{claimproof}
			Let $\eta \in \Ker\varphi\setminus\{1\}$. Then $g^1_3 \circ \eta = g^1_3$. Since $\deg g^1_3=3$, it follows that $g^1_3$ is Galois and $\Ker\varphi\cong C_3$. 
	\end{claimproof}
	
	By the classification of finite subgroups of $\Aut(\mathbb{P}^1)$, $\img \varphi$ is isomorphic to $C_m$, $D_m$, $A_4$, $S_4$, or $A_5$. If $\img \varphi \cong A_5$, then $|G| = 3\cdot 60 = 180$. However, this value is excluded by the restrictions arising in the proof of Hurwitz's theorem (see \cite[V.1.3, Theorem (Hurwitz)]{FarkasKra1992}). (Alternatively, one may use the bound $|\Aut(C)|\le 120$ for curves of genus~4; see \cite[p.~47, Exercise~F--16]{ACGH1985}.) Therefore, $\img \varphi \cong C_m$, $D_m$, $A_4$, or $S_4$.	
	 
	Let $l$ be an $S_3$-line. Then, we may assume that $G_l = \langle \sigma, \tau \rangle$, where $\sigma$ and $\tau$ are the matrices \eqref{eq:representation-of-S_3} in Proposition~\ref{prop:generator-of-Gal-S_3-line}. Let $Q=0$ and $F=0$ be the defining equations of $C$, where $Q$ is a quadratic form and $F$ is a cubic form. Since $\sigma^* Q = c Q$ and $\tau^* Q = c' Q$ for some $c,c' \in k\setminus \{0\}$, we may assume that $Q=X^2-ZW$. Let $\eta \in \Aut(C)$ be a generator of the Galois group associated to $g^1_3$, i.e., $g^1_3 \circ \eta = g^1_3$ and $\ord(\eta) = 3$. The unique trigonal morphism $g^1_3: C \rightarrow \mathbb{P}^1$ is given by the projection from the vertex $(0:1:0:0)$ of $Q=0$, that is, $g^1_3: C \rightarrow (X^2-ZW=0) \subset \mathbb{P}^2$, $(X:Y:Z:W) \mapsto (X:Z:W)$. Note that the plane quadric curve $X^2 -ZW =0$ is isomorphic to $\mathbb{P}^1$.
	 
	\begin{claim} \label{claim:representation-of-eta}
		The representation matrix of $\eta$ is expressed as 
			\[ \begin{pmatrix}
				1 & 0 & 0 & 0 \\
				a & b & c & d \\
				0 & 0 & 1 & 0 \\
				0 & 0 & 0 & 1 \\
			\end{pmatrix},  \]
		where $a,b,c,d \in k$, $b^3=1$ and $b \ne 1$. 
	\end{claim}
	\begin{claimproof}
		Let $(a_{ij})$ be a representation matrix of $\eta$. Namely, 
			\[
				(X:Y:Z:W) \mapsto (X':Y':Z':W'), %\quad 
				{}^t(X':Y':Z':W') = (a_{ij})\, {}^t(X:Y:Z:W).
			\] 
			From $g^1_3 \circ \eta = g^1_3$, we have that $(X:Z:W)=(X':Z':W')$ for every point $(X:Y:Z:W) \in C$. Since $I(C)=(Q,F)$, 
			\[
			XZ'-ZX', \, ZW'-WZ', \, WX'-XW' \in (Q).  
			\]  
			From $ZW'-WZ' = Z(a_{41}X+a_{42}Y+a_{43}Z+a_{44}W)-W(a_{31}X+a_{32}Y+a_{33}Z+a_{34}W)$ and $Q=X^2-ZW$, we see that $ZW'-WZ' = 0$ identically. We have that $a_{33} = a_{44}$ and $a_{ij} = 0$ where $i \ne j$, $i \in \{3,4\}$ and $j \in \{1,2,3,4\}$. From $XZ'-ZX' = X(a_{33}Z)-Z(a_{11}X+a_{12}Y+a_{13}Z+a_{14}W)$ and $Q=X^2-ZW$, we see that $XZ'-ZX' =  0$ identically. We have that $a_{11}=a_{33}$ and $a_{12}=a_{13}=a_{14}=0$. Since the order of $\eta$ equals $3$, we conclude the Claim~\ref{claim:representation-of-eta}. 
	\end{claimproof}
		
	Let $P_1, \dots, P_6$ be all the ramification points of $g^1_3$. 
		By Lemma 3.3 in \cite{KomedaTakahashi2022}, these 6 points are on a plane $H$. Since $\eta$ is represented by the matrix above, the plane $H$ is given by $aX+(b-1)Y+cZ+dW = 0$. Since the six points $P_1, \dots, P_6$ are determined by the unique cyclic trigonal pencil $g^1_3$, the plane $H$ is determined by $C$. Hence, $\sigma (H) = H$ and $c = d = 0$. We have that 
	\begin{equation} \label{eq:commutativity}
		\eta \sigma = \sigma \eta \text{ and } \eta \tau = \tau \eta. 
	\end{equation}
	Since $c=d=0$, by taking a suitable projective transformation preserving the matrices \eqref{eq:representation-of-S_3} and $Q=X^2-ZW$, we may assume that 
	\begin{equation} \label{eq:eta}
		\eta = \begin{pmatrix}
			1 & 0 & 0 & 0 \\
			0 & \omega & 0 & 0 \\
			0 &	0 & 1 & 0 \\
			0 & 0 & 0 & 1 \\
		\end{pmatrix}, 
	\end{equation}
	where $\omega$ is a primitive cubic root of unity. 
	
	\begin{claim} \label{claim:varphi(Gl)=S3}
		$\varphi(G_l) \cong S_3$.
	\end{claim}
	\begin{claimproof}
		Suppose that $\varphi(G_l) \not\cong S_3$. Then $\langle \eta \rangle \subset G_l$. We have that $\sigma = \eta$ or $\eta^2$. The equation $\eta \tau = \tau \eta$ contradicts $\sigma \tau = \tau \sigma^2 \ne \tau\sigma$. 
	\end{claimproof}
	
	We note that Equation~\eqref{eq:commutativity} and Claim~\ref{claim:varphi(Gl)=S3} hold for an arbitrary $S_3$-line.
	
	\begin{claim} \label{claim:varphi(Gl)-ne-varphi(Gl')}
		Let $l'$ be an $S_3$-line with $l \ne l'$. Let $\sigma'$ be an element in $G_{l'}$ whose order equals $3$. Then $\langle \varphi(\sigma) \rangle \ne \langle \varphi(\sigma') \rangle$, in particular, $\varphi(G_l) \ne \varphi(G_{l'})$.
	\end{claim}
	\begin{claimproof}
		Suppose that $\langle \varphi(\sigma) \rangle = \langle \varphi(\sigma') \rangle$. We may assume that $ \varphi(\sigma)  =  \varphi(\sigma') $. Then, $\sigma^{-1} \sigma' \in \langle \eta \rangle$. 
		Namely, $\sigma'=\eta \sigma$ or $\eta^2 \sigma$. Assume that 
		\[
			\sigma' = \eta \sigma = 
				\begin{pmatrix}
					1 & 0 & 0 & 0 \\
					0 & \omega & 0 & 0 \\
					0 & 0 & \omega & 0 \\
					0 & 0 & 0 & \omega^2
				\end{pmatrix}. 
		\]
			
		Let $(t_{ij})$ be the representation matrix of $\tau'$. Using the note preceding Proposition~\ref{prop:distinct_l_induce_distinct_G_l}, the $S_3$-line $l'$ is given by $Y=Z=0$. Using Lemma~\ref{lemma:representation-when-l:X=Y=0}, $t_{22}=t_{33}$ and $t_{ij}=0$ where $i \ne j$, $i \in \{2,3\}$ and $j \in \{1,2,3,4\}$. We may assume that $t_{22}=t_{33}=1$. Since $\tau' \eta = \eta \tau'$, we see that $t_{12}=t_{42}=0$. Since $\sigma' \tau' = \tau' \sigma'^2$, we have that $t_{11}=t_{13}=t_{43}=t_{44}=0$. Since $\tau'^2=\id$, 
		\[
			\tau'=
			\begin{pmatrix}
				0 & 0 & 0 & \alpha \\
				0 & 1 & 0 & 0 \\
				0 & 0 & 1 & 0 \\
				\alpha^{-1} & 0 & 0 & 0 \\
			\end{pmatrix}, 
		\] 
		where $\alpha \in k\setminus\{0\}$. 
			
		However, since $Q=X^2-ZW$, the automorphism $\tau'$ does not satisfy the condition $\tau'^*Q = cQ$ where $c \in k \setminus \{0\}$, this is a contradiction. In the case that $\sigma' = \eta^2 \sigma$, by a similar argument, we also have a contradiction. We conclude that $\langle \varphi(\sigma) \rangle \ne \langle \varphi(\sigma') \rangle$. 
	\end{claimproof}
			
	The number of $C_3$ subgroups contained in $S_3$ subgroups in $C_m$ (resp.\ $D_m$, $A_4$, $S_4$) is equal to $0$ (resp.\ at most $1$, equal to $0$, equal to $4$). We conclude that the number of $S_3$-lines is at most $4$. 
\end{proof}

%%%%%
\subsection{The number of $S_3$-lines when $\rank Q=4$}

Assume that $\rank Q=4$, so that $C$ admits exactly two trigonal pencils $g^1_3$ and $h^1_3$. Let
\[
	O_3:=\{\,\sigma\in\Aut(C)\mid \ord(\sigma)=3 \text{ and } \sigma\in G_l
	\text{ for some $S_3$-line } l\,\}.
\]
In order to show that the number of $S_3$-lines is at most $10$, it suffices to prove that $\# O_3\le 20$.

For each $\sigma\in O_3$, both $\sigma^*g^1_3$ and $\sigma^*h^1_3$ are trigonal pencils. Since $\ord(\sigma)=3$ and there exist exactly two trigonal pencils on $C$, it follows that
\[
	\sigma^*g^1_3\simeq g^1_3
	\quad\text{and}\quad
	\sigma^*h^1_3\simeq h^1_3 .
\]
Let $H$ be the subgroup of $\Aut(C)$ generated by $O_3$. Then every element of $H$ preserves both pencils $g^1_3$ and $h^1_3$ up to equivalence.

Applying Lemma~\ref{lemma:action-on-trigonal} to each trigonal pencil, we obtain homomorphisms
\[
	\varphi_g,\varphi_h:H\longrightarrow\Aut(\mathbb P^1),
\]
and the associated short exact sequences
\begin{equation}\label{eq:exact-seq-two-trigonal}
	\begin{aligned}
		1&\longrightarrow\Ker\varphi_g\longrightarrow H
		\xrightarrow{\ \varphi_g\ }\img\varphi_g\longrightarrow 1,\\
		1&\longrightarrow\Ker\varphi_h\longrightarrow H
		\xrightarrow{\ \varphi_h\ }\img\varphi_h\longrightarrow 1.
	\end{aligned}
\end{equation}

\begin{proposition}
	Assume that $\rank Q = 4$, and at least one of $\Ker \varphi_g$ and $\Ker \varphi_h$ is trivial. Then, the number of $S_3$-lines is at most $10$.  
\end{proposition}

\begin{proof}
	Assume that $\Ker \varphi_g = 1$. Then, $H \cong \img \varphi_g \subset \Aut(\mathbb{P}^1)$. By the classification of finite subgroups of $\Aut(\mathbb{P}^1)$, $H \cong C_m$, $D_m$, $A_4$, $S_4$, or $A_5$. The number of elements with order three in $C_m$ (resp.\ $D_m$, $A_4$, $S_4$, $A_5$) is at most two (resp.\ at most two, equal to $8$, equal to $8$, equal to $20$). We have that $\# O_3 \leq 20$. 
\end{proof}

\begin{proposition} \label{prop:two_trigonal_are_Gal_S3Lines}
	Assume that $\rank Q = 4$, and that both $\Ker \varphi_g$ and $\Ker \varphi_h$ are non-trivial. Then the number of $S_3$-lines is at most $2$.  
\end{proposition}

\begin{proof}
	For $\eta \in \Ker\varphi_g$, we have that $g^1_3 \circ \eta = g^1_3$. Hence, $g^1_3$ is a Galois covering. Similarly, $h^1_3$ is also a Galois covering. By Corollary~3.3 in \cite{KomedaTakahashi2021}, we may assume that $C$ is given by the equations
	\[ 
		Q=XW-YZ=0 \text{ and } 
		F= Z(W-Z)(W+Z)-Y^3-cXY^2+9X^2Y+cX^3=0, 
	\]
	where $c \in k$. Moreover, 
	\[g^1_3: (X:Y:Z:W) \longmapsto (X:Y), \qquad h^1_3: (X:Y:Z:W) \longmapsto (X:Z),\] 
	and the Galois groups associated to $g^1_3$ and $h^1_3$ are generated by 
	\begin{equation}
		\eta_g :=\begin{pmatrix}
			1 & 0 & 0 & 0 \\
			0 & 1 & 0 & 0 \\
			0 & 0 & \omega & 0 \\
			0 & 0 & 0 & \omega \\
		\end{pmatrix} 
		\qquad \text{and} \qquad
		\eta_h:=\begin{pmatrix}
			1 & 1 & 0 & 0 \\
			-3 & 1 & 0 & 0 \\
			0 & 0 & 1 & 1 \\
			0 & 0 & -3 & 1 \\
		\end{pmatrix}, 
	\end{equation}
	respectively. 
	
	Let $l_1$, $l_2$, $l_3$ and $l_4$ be lines 
	\[X=Y=0, \qquad Z=W=0, \qquad -\sqrt{-3}X+Y=-\sqrt{-3}Z+W =0, \qquad \text{and}\]
	\[\sqrt{-3}X + Y=\sqrt{-3}Z + W =0,\] 
	respectively. We have that 
	\[\Fix(\eta_g) := \{P \in \mathbb{P}^3 \mid \eta_g(P) = P\} = l_1 \cup l_2 \qquad \text{and} \qquad \Fix(\eta_h) = l_3 \cup l_4.\] 
	The number of all ramification points of $g^1_3$ (resp.\ $h^1_3$) equals $6$, and the $3$ points of them are on $l_1$ (resp.\ $l_3$) and other $3$ points are on $l_2$ (resp.\ $l_4$).  
	
	Since $C$ has exactly two trigonal pencils $g^1_3$ and $h^1_3$, for any $\sigma \in \Aut(C)$, we have that $\sigma (l_1 \cup l_2) = l_1 \cup l_2$ or $\sigma (l_1 \cup l_2) = l_3 \cup l_4$. For each $\sigma \in O_3$, since $\ord \sigma =3$, we have that $\sigma(l_1 \cup l_2) = l_1 \cup l_2$. Hence, since $\ord \sigma = 3$ again, we have that $\sigma(l_1)=l_1$ and $\sigma(l_2)=l_2$. Similarly, $\sigma(l_3)=l_3$ and $\sigma(l_4) = l_4$. 
	
	Let $\sigma \in O_3$. Then, since $\sigma(l_i) = l_i$ ($i=1,2$), $\sigma$ must be represented by a matrix  
	\[\begin{pmatrix}
		a_{11} & a_{12} & 0 & 0 \\
		a_{21} & a_{22} & 0 & 0 \\
		0 & 0 & b_{11} & b_{12} \\
		0 & 0 & b_{21} & b_{22} \\
	\end{pmatrix}, \]
	where $a_{i j}, b_{ij} \in k$ ($i,j \in \{1,2\}$). Let $A$ and $B$ be the $2 \times 2$ matrices $(a_{ij})$ and $(b_{ij})$. Since $\sigma ^* Q = a Q$ for some $a \in k \setminus \{0\}$, checking the coefficients of $XZ$, $YW$, $XW$, and $YZ$ of both sides, we see that $A = \alpha B$ for some $\alpha \in k \setminus \{0\}$. Namely, 
	\[\sigma = \begin{pmatrix}
		\alpha B & O \\
		O & B
	\end{pmatrix}. \]
	
	Since $\sigma$ acts on the set 
	\[l_1 \cap C = \{ (0:0:0:1), (0:0:1:1), (0:0:-1:1) \}, \] 
	we see that $B$ is  
	\[ b_1 \begin{pmatrix} 1 & 0 \\ 0 & 1 \end{pmatrix}, b_2 \begin{pmatrix} 1 & 1 \\ -3 & 1 \end{pmatrix}, b_3 \begin{pmatrix} 1 & -1 \\ 3 & 1 \end{pmatrix}, b_4 \begin{pmatrix} -1 & 0 \\ 0 & 1 \end{pmatrix}, b_5 \begin{pmatrix} -1 & 1 \\ 3 & 1 \end{pmatrix}, \text{ or } b_6 \begin{pmatrix} 1 & 1 \\ 3 & -1 \end{pmatrix},\]
	where $b_i \in k\setminus \{0\}$ ($i=1, \dots, 6$). Since $\sigma$ has order $3$, we see that 
	\[b_4 \begin{pmatrix} -1 & 0 \\ 0 & 1 \end{pmatrix}, b_5 \begin{pmatrix} -1 & 1 \\ 3 & 1 \end{pmatrix}, \text{ and } b_6 \begin{pmatrix} 1 & 1 \\ 3 & -1 \end{pmatrix}\]
	are unsuitable for $B$. Moreover, we see that 
	\begin{equation}\label{eq:O3-subgroup} 
		O_3 \subset 
		\Biggl\langle 
			\begin{pmatrix} 1 & 0 & 0 & 0 \\ 0 & 1 & 0 & 0 \\ 0 & 0 & \omega & 0 \\ 0 & 0 & 0 & \omega \end{pmatrix}, 
			\begin{pmatrix} 1 & 1 & 0 & 0 \\ -3 & 1 & 0 & 0 \\ 0 & 0 & 1 & 1 \\ 0 & 0 & -3 & 1 \end{pmatrix} 
			\Biggr\rangle \cong C_3 \times C_3 . 
	\end{equation} 
	By Proposition~\ref{prop:generator-of-Gal-S_3-line}, every element of $O_3$ has eigenspaces of dimensions $1$, $1$, and $2$. Among the elements of the right-hand side of \eqref{eq:O3-subgroup}, exactly four have this property. Hence $\#O_3 \leq 4$. This means that the number of $S_3$-lines is at most $2$.
\end{proof}

%%%%%
\subsection{The number of $K_4$-lines when $\rank Q=3$}

Assume that $\rank Q=3$. Then $C$ admits a unique trigonal morphism $g^1_3$. Let $G$ be the subgroup of $\Aut(C)$ generated by the elements
\[
	\{\sigma\in\Aut(C)\mid \sigma\in G_l \text{ for some $K_4$-line } l\}.
\]
Since every element of $G$ preserves the trigonal morphism $g^1_3$ up to equivalence, Lemma~\ref{lemma:action-on-trigonal} yields a homomorphism
\[
	\varphi:G\longrightarrow\Aut(\mathbb P^1)
\]
and the associated short exact sequence
\begin{equation}
	1\longrightarrow\Ker\varphi\longrightarrow G
	\xrightarrow{\ \varphi\ }\img\varphi\longrightarrow 1.
\end{equation}

\begin{proposition}
	Assume that $\rank Q = 3$. Then, the number of $K_4$-lines is at most $5$.  
\end{proposition}

\begin{proof}
We assume that $\Ker \varphi = 1$. Then, $G \cong \img \varphi \subset \Aut(\mathbb{P}^1)$. By the classification of finite subgroups of $\Aut(\mathbb{P}^1)$, $G \cong C_m$, $D_m$, $A_4$, $S_4$, or $A_5$. By Proposition~\ref{prop:two-G_ls-belonging-to-K_4-lines}, we note that $G_l \ne G_{l'}$ for two distinct $K_4$-lines $l$ and $l'$. The number of $K_4$ subgroups of $C_m$ (resp.\ $D_m$, $A_4$, $S_4$, $A_5$) equals $0$ (resp.\ $0$ or $m/2$, $1$, $4$, $5$). Hence, if $\img \varphi \not\cong D_m$, or $\img \varphi \cong D_m$ where $m$ is odd, then the number of $K_4$-lines is at most $5$. Suppose that $G \cong \img \varphi \cong D_m$ and $m$ is even. Let $G = \langle \sigma, \tau \rangle$ where $\sigma^m = e$ and $\sigma \tau = \tau \sigma^{m-1}$. Then, $K_4$ subgroups of $G$ are $\{e, \sigma^{m/2}, \tau \sigma^l, \tau \sigma^{l+m/2} \}$, where $l=0, \dots, m/2-1$. Note that every $K_4$ subgroup contains $\sigma^{m/2}$. Let $l$ be a $K_4$-line. By Proposition~\ref{prop:generator-of-Gal-K4-line}, we may assume that $G_l$ is the group \eqref{eq:standard-representation-K4}. If 
	\[	\sigma^{m/2} = 
			\begin{pmatrix}
				1 & 0 & 0 & 0 \\
				0 & 1 & 0 & 0 \\
				0 & 0 & -1 & 0 \\
				0 & 0 & 0 & -1 \\	 
			\end{pmatrix},
	\] 
	then the number of $K_4$-lines equals $1$. Indeed, if there exists a $K_4$-line $l'$ other than $l$, then $\sigma^{m/2} \in G_{l'}$ holds. This contradicts Lemma~\ref{lemma:fixed-points-of-trace-0-involution}. 
	Assume \[	\sigma^{m/2} = 
			\begin{pmatrix}
				1 & 0 & 0 & 0 \\
				0 & 1 & 0 & 0 \\
				0 & 0 & -1 & 0 \\
				0 & 0 & 0 & 1 \\	 
			\end{pmatrix}.
	\] 
	Consider the morphism $\pi_l: C \rightarrow \mathbb{P}^1$. The degree $4$ morphism $\pi_l$ can be decomposed as $C \rightarrow C/\langle \sigma^{m/2} \rangle \rightarrow \mathbb{P}^1$. Let $P_1, \dots, P_6$ be ramification points of $C \rightarrow C/\langle \sigma^{m/2} \rangle$. We may assume that $\pi_l(P_1) = \pi_l(P_2)$, $\pi_l(P_3) = \pi_l(P_4)$, and $\pi_l(P_5) = \pi_l(P_6)$, i.e., $2P_1 + 2P_2$, $2P_3 + 2P_4$ , and $2P_5 + 2P_6$ are fibers of $\pi_l$. Since the canonical divisor $K_C$ is cut out by a hyperplane, and since $\deg(K_C-2P_1-2P_2)=2$ and $C$ is non-hyperelliptic, 
		\[ 
			\dim_k H^0\bigl(C,\mathcal O_C(K_C-2P_1-2P_2)\bigr)=1. 
		\] 
		Hence, by the Riemann--Roch theorem, 
		\[ 
			\dim_k H^0\bigl(C,\mathcal O_C(2P_1+2P_2)\bigr)=2. 
		\]
	This means that $\pi_l$ is a morphism given by the complete linear system $|2P_1 + 2P_2|$, where $P_1$ and $P_2$ are two of the $6$ ramification points of $C \rightarrow C/\langle \sigma^{m/2} \rangle$. If there exists another $K_4$-line $l'$, then $\pi_{l'}$ is also a morphism given by $|2P_1+ 2P_i|$ for some $P_i \in \{ P_2, \dots P_6\}$. Therefore, the number of $K_4$-lines is at most $5$. 
	
	We assume that $\Ker \varphi \ne 1$. Then, $\Ker \varphi \cong C_3$, since $\Ker \varphi \subset \{ \sigma \in \Aut(C) \mid  g^1_3 = g^1_3 \circ \sigma \} \cong 1$ or $C_3$. Note that the trigonal morphism $g^1_3$ is Galois. 
	\[
		\xymatrix{
  			C \ar[r]^{\sigma} \ar[d]_{g^1_3}^{\hspace{3.3ex} \circlearrowleft} & C \ar[d]^{g^1_3} \\
  			\mathbb{P}^1 \ar[r]^{\mathrm{id}} & \mathbb{P}^1
		}
	\]
	
	\begin{claim} \label{claim:varphi(G_l)-cong-K_4}
		Let $l$ be a $K_4$-line. Then $\varphi(G_l) \cong K_4$.  
		If $l'$ is another $K_4$-line, then $\varphi(G_l) \neq \varphi(G_{l'})$.
	\end{claim}
	
	\begin{claimproof}  
		Since $\Ker \varphi \cong C_3$ and $G_l \cong K_4$, we have $\Ker \varphi \, \cap\,  G_l = 1$. Hence $\varphi(G_l) \cong K_4$.  Suppose, for contradiction, that $\varphi(G_l) = \varphi(G_{l'}) \subset \Aut(\mathbb{P}^1)$.  Then there exist involutions $\sigma_1 \in G_l$ and $\sigma_2 \in G_{l'}$ such that $\sigma_1 \neq e$, $\sigma_2 \neq e$, $\sigma_1 \neq \sigma_2$, and $\varphi(\sigma_1) = \varphi(\sigma_2)$. Let $\eta$ be a generator of $\Ker \varphi \cong C_3$. Then $\sigma_1 \sigma_2 (= \sigma_1 \sigma_2^{-1} ) = \eta$ or $\eta^2$. We may assume that $\sigma_1 \sigma_2 = \eta$. Since $\eta (\sigma_2 \sigma_1) = \sigma_1 \sigma_2 \sigma_2 \sigma_1 = e$, we have $\sigma_2 \sigma_1 = \eta^{-1} = \eta^2$. Moreover,
		\[
			\sigma_1 \eta = \eta^2 \sigma_1.
		\]
		Indeed,
		$
			\sigma_1 \eta = \sigma_1 \sigma_1 \sigma_2 = \sigma_2 = \sigma_2 \sigma_1 \sigma_1 = \eta^2 \sigma_1.
		$
		By taking a projective transformation, we may diagonalize $\eta$. Then
		\begin{equation} \label{eq:diagonal-representation-of-eta}
			\eta =
			\begin{pmatrix}
				1 & 0 & 0 & 0 \\ 0 & 1 & 0 & 0 \\ 0 & 0 & \omega & 0 \\ 0 & 0 & 0 & \omega
			\end{pmatrix}, \quad
			\begin{pmatrix}
				1 & 0 & 0 & 0 \\ 0 & 1 & 0 & 0 \\ 0 & 0 & \omega & 0\\ 0 & 0 & 0 & \omega^2
			\end{pmatrix},\ \text{or}\ 
			\begin{pmatrix}
				1 & 0 & 0 & 0\\ 0 & 1 & 0 & 0 \\ 0 & 0 & 1 & 0 \\ 0 & 0 & 0 & \omega
			\end{pmatrix},
		\end{equation}
		where $\omega$ is a primitive cubic root of unity. The cyclic morphism $g^1_3$ has six ramification points.  By Proposition~\ref{prop:well-known-facts}, the trigonal morphism $g^1_3$ is given by the projection from the vertex of the quadric cone $Q=0$. No three of the six ramification points are collinear. Indeed, a line through the vertex of the quadric cone $Q=0$ contains at most one of them, while any other line meets $Q=0$ in at most two points. Considering the fixed locus $\Fix(\eta) = \{ P \in \mathbb{P}^3 \mid \eta(P) = P \}$, we see that $\eta$ cannot be the first two matrices in~\eqref{eq:diagonal-representation-of-eta}. Thus, $\eta$ must be the third one in~\eqref{eq:diagonal-representation-of-eta}, and the six ramification points of $g^1_3$ lie on the plane $W=0$. The vertex of $Q=0$ is $(0:0:0:1)$. Since $Q=0$ is the unique quadric surface containing $C$, the automorphism $\sigma_1$ also fixes the vertex and acts on the set of six ramification points. Therefore, $\sigma_1((0:0:0:1)) = (0:0:0:1)$ and $\sigma_1(\{W=0\}) = \{W=0\}$; that is,
		\[
			\sigma_1 =
			\begin{pmatrix}
				* & * & * & 0 \\ * & * & * & 0 \\ * & * & * & 0 \\ 0 & 0 & 0 & *
			\end{pmatrix}.
		\]
		Now we have $\sigma_1 \eta = \eta \sigma_1$, which contradicts that $\sigma_1 \eta = \eta^2 \sigma_1$ and $\eta \neq e$. This completes the proof of Claim~\ref{claim:varphi(G_l)-cong-K_4}.
	\end{claimproof}
	
	The finite group $\operatorname{Im}\varphi \subset \operatorname{Aut}(\mathbb{P}^1)$ is isomorphic to one of $C_m$, $D_m$, $A_4$, $S_4$, or $A_5$. 
	We first note that $\operatorname{Im}\varphi \not\cong A_5$. Indeed, if $\img \varphi \cong A_5$, then $|G| = 3\cdot 60 = 180$. However, this value is excluded by the restrictions arising in the proof of Hurwitz's theorem (see \cite[V.1.3, Theorem (Hurwitz)]{FarkasKra1992}). (Alternatively, one may use the bound $|\Aut(C)|\le 120$ for curves of genus~4; see \cite[p.~47, Exercise~F--16]{ACGH1985}.) 	
	
	By counting the number of $K_4$ subgroups in the groups $C_m$, $D_m$, $A_4$, and $S_4$, and using Claim~\ref{claim:varphi(G_l)-cong-K_4}, we find that the number of $K_4$-lines is at most four, except in the case where $\operatorname{Im}\varphi \cong D_m$ with $m$ even. 

	Henceforth, we assume that $\operatorname{Im}\varphi \cong D_m$ and that $m$ is even. Let $\img \varphi = \langle \sigma, \tau \rangle$, where $\sigma^m=e$, $\tau^2=e$, and $\sigma\tau = \tau \sigma^{m-1}$. All $K_4$ subgroups in $\img \varphi$ are  
	\[ 
		\{e, \sigma^{m/2}, \tau \sigma^l, \tau \sigma^{l+m/2}\}, 
	\]
	where $l=0, \dots, m/2-1$. Note that the element $\sigma^{m/2}$ is contained in every $K_4$ subgroup. By the argument used in the proof of Claim~\ref{claim:varphi(G_l)-cong-K_4}, we see that there exists a common involution $\rho$ contained in all $G_l$, where $l$ is a $K_4$-line. Let $l$ and $l'$ be two distinct $K_4$-lines. Both projections $\pi_l$ and $\pi_{l'}$ factor through the quotient morphism $q: C \rightarrow C/\langle \rho \rangle$. 
	\[
  		\xymatrix{
    		& C \ar[ld]_{\pi_{l}} \ar[d]^q \ar[rd]^{\pi_{l'}} & \\
    	\mathbb{P}^1 & C/\langle \rho \rangle \ar[l] \ar[r] & \mathbb{P}^1
  		}
	\]
	By the proof of Lemma~\ref{lemma:fixed-points-of-trace-0-involution}, the genus of $C/\langle \rho \rangle$ is $1$. The quotient morphism $q$ has six ramification points, denoted by $P_1, \dots, P_6$. Let $P_1$ and $P_2$ be two of these six ramification points such that $\pi_{l}(P_1) = \pi_{l}(P_2)$; namely, $2P_1 + 2P_2$ is a fiber of $\pi_l$. Then $ \dim_k H^0\bigl(C,\mathcal O_C(2P_1+2P_2)\bigr)\ge 2$. Since $\deg(K_C-2P_1-2P_2)=2$ and $C$ is non-hyperelliptic, $\dim_k H^0\bigl(C,\mathcal O_C(K_C-2P_1-2P_2)\bigr)\le 1$. The Riemann--Roch theorem then gives $\dim_k H^0\bigl(C,\mathcal O_C(2P_1+2P_2)\bigr)=2$. Hence $\pi_l = \Phi_{|2P_1 + 2P_2|}$. For $l'$, there exists a ramification point $P_i$ of $q$ such that $\pi_{l'} = \Phi_{|2P_1 + 2P_i|}$. Therefore, the number of $K_4$-lines is at most five.
\end{proof}

%%%%%
\subsection{The number of $K_4$-lines when $\rank Q=4$}

Assume that $\rank Q=4$, so that $C$ admits exactly two trigonal pencils $g^1_3$ and $h^1_3$. Let $l$ be a $K_4$-line. By Proposition~\ref{prop:generator-of-Gal-K4-line}, after composing with suitable projective transformations, we may assume that $l:X=Y=0$ and that the group $G_l$ is given by the standard representation \eqref{eq:standard-representation-K4}. Let
\[
	\sigma=
	\begin{pmatrix}
		1 & 0 & 0 & 0\\
		0 & 1 & 0 & 0\\
		0 & 0 & -1 & 0\\
		0 & 0 & 0 & 1
	\end{pmatrix},
	\qquad
	\tau=
	\begin{pmatrix}
		1 & 0 & 0 & 0\\
		0 & 1 & 0 & 0\\
		0 & 0 & 1 & 0\\
		0 & 0 & 0 & -1
	\end{pmatrix}.
\]

Note that the product $\sigma\tau$ is an involution of trace $0$ contained in $G_l$.

\begin{lemma}\label{lemma:actions-for-trigonal-maps}
	The following relations hold:
	\[
		g^1_3\circ\sigma\simeq h^1_3,\quad
		h^1_3\circ\sigma\simeq g^1_3,\quad
		g^1_3\circ\tau\simeq h^1_3,\quad
		h^1_3\circ\tau\simeq g^1_3,
	\]
	and
	\[
		g^1_3\circ(\sigma\tau)\simeq g^1_3,\quad
		h^1_3\circ(\sigma\tau)\simeq h^1_3.
	\]
\end{lemma}

\begin{proof}
	By Proposition~\ref{prop:well-known-facts}, the quadric $V(Q)$ is the unique quadric containing $C$. Since $\sigma$ and $\tau$ preserve $V(Q)$, the quadratic form $Q$ can be written as
	\[
		Q=A(X,Y)B(X,Y)-(Z-dW)(Z+dW)
	\]
	for some linear forms $A,B$ and $d\in k$.
	As $\rank Q=4$, after suitable projective transformations preserving the line $l:X=Y=0$ and the representation of $G_l$, we may assume that
	\[
		Q=XY-(Z-dW)(Z+dW).
	\]

	Via the Segre embedding and a linear change of coordinates, we obtain an isomorphism $V(Q)\cong\mathbb P^1\times\mathbb P^1$.
	Under this identification, the morphisms $g^1_3$ and $h^1_3$ are given by the two projections, and may be written as
	\[
		g^1_3:(X:Y:Z:W)\longmapsto (X:Z-dW),\qquad
		h^1_3:(X:Y:Z:W)\longmapsto (X:Z+dW).
	\]
	The assertions follow by a direct computation.
\end{proof}

Let $H$ be the subgroup of $\Aut(C)$ generated by all involutions of trace $0$ contained in $G_l$ for some $K_4$-line $l$. By Lemma~\ref{lemma:actions-for-trigonal-maps}, every element of $H$ preserves both trigonal pencils $g^1_3$ and $h^1_3$ up to equivalence.
Applying Lemma~\ref{lemma:action-on-trigonal} to each pencil, we obtain homomorphisms
\[
	\varphi_g,\varphi_h:H\longrightarrow\Aut(\mathbb P^1)
\]
and the associated short exact sequences
\begin{equation}\label{eq:exact-seq-of-H}
	\begin{aligned}
		1&\longrightarrow\Ker\varphi_g\longrightarrow H
		\xrightarrow{\ \varphi_g\ }\img\varphi_g\longrightarrow 1,\\
		1&\longrightarrow\Ker\varphi_h\longrightarrow H
		\xrightarrow{\ \varphi_h\ }\img\varphi_h\longrightarrow 1.
	\end{aligned}
\end{equation}

\begin{proposition}
	Assume that $\rank Q = 4$, and one of $\Ker \varphi_g$ and $\Ker \varphi_h$ of the short exact sequence \eqref{eq:exact-seq-of-H} is trivial. Then, the number of $K_4$-lines is at most $15$.  
\end{proposition}

\begin{proof}
	We may assume that $\Ker \varphi_g=1$. Let $l_i$ ($i=1,2, \dots $) be mutually distinct $K_4$-lines, and $\sigma_i \in G_{l_i}$ ($i=1,2, \dots $) be elements with trace $0$. Using Proposition~\ref{prop:two-G_ls-belonging-to-K_4-lines} and Lemma~\ref{lemma:actions-for-trigonal-maps}, we see that $\sigma_1, \sigma_2, \dots$ are mutually distinct, and all belong to $H$. 
	By the assumptions, $H \cong \img \varphi_g \cong C_m$, $D_m$, $A_4$, $S_4$, or $A_5$. The number of elements of order $2$ in $C_m$ (resp.\ $D_m, A_4, S_4, A_5$) is at most $1$ (resp.\ at most $m$ or $m+1$, equal to  $3$, equal to $9$, equal to $15$). Hence, if $H \not \cong D_m$, the number of $K_4$-lines is at most $15$.

	Assume that $H \cong D_m$ and $H = \langle \sigma, \tau \rangle$, where $\sigma$ has order $m$, $\tau$ has order $2$, and $\sigma \tau = \tau \sigma^{-1}$. We show that $m \le 12$. Since $g^1_3 \circ \alpha = \varphi_g(\alpha) \circ g^1_3$ for any $\alpha \in H$, the group $H$ acts on the branch locus of $g^1_3$. The automorphism
	$
	\varphi_g(\sigma) \colon \mathbb{P}^1 \to \mathbb{P}^1
	$
	has exactly two fixed points, say $Q_1$ and $Q_2$. Since $g^1_3$ has at least six branch points, there exists a point $P \in \mathbb{P}^1$ such that $P$ is a branch point of $g^1_3$ and $\varphi_g(\sigma)(P) \ne P$.
	Consider the points $\varphi_g(\sigma)^i(P)$ for $i = 1, \dots, m$. Since an automorphism of $\mathbb{P}^1$ is determined by the images of three points, namely the images of $Q_1$, $Q_2$, and $P$, if $\varphi_g(\sigma)^i(P) = P$, then $\varphi_g(\sigma)^i = \mathrm{id}$. Since $g^1_3$ has at most twelve branch points, the number of distinct points among $\varphi_g(\sigma)^i(P)$ ($i = 1, \dots, m$) is at most twelve. Hence the order of $\varphi_g(\sigma)$, that is, $m$, is at most $12$. Since the number of involutions in $H \cong D_m$ with $m \le 12$ is at most $13$, the number of $K_4$-lines is also at most $13$.
	\end{proof}

\begin{proposition} \label{prop:two_trigonal_are_Gal_K4Lines}
	Assume that $\rank Q = 4$, both $\Ker \varphi_g$ and $\Ker \varphi_h$ of the short exact sequence \eqref{eq:exact-seq-of-H} are non-trivial. Then, the number of $K_4$-lines is at most $9$.  
\end{proposition}

\begin{proof}
	For a $\sigma \in \Ker \varphi_g$, $g^1_3 \circ \sigma = g^1_3$. By the assumption $\Ker \varphi_g \ne 1$, we have that $g^1_3$ is Galois. Similarly, $h^1_3$ is also Galois. 
	From \cite[Theorem 1.2]{KomedaTakahashi2021}, we may assume that $I(C) = (Q, F)$ where 
	\begin{equation}
		Q=XW-YZ \text{ and } F=Z(W-Z)(W+Z) - (Y^3+cXY^2-9X^2Y-cX^3),
	\end{equation}
	($c \in k$). The Galois groups of $g^1_3$ and $h^1_3$ are $\Gal(g^1_3):=\{ \sigma \in \Aut(C) \mid g^1_3 \circ \sigma =g^1_3  \} = \langle \eta_g  \rangle$ and $\Gal(h^1_3):=\{ \sigma \in \Aut(C) \mid h^1_3 \circ \sigma =h^1_3  \} = \langle \eta_h  \rangle$, respectively, where
	\begin{equation}
		\eta_g= 
			\begin{pmatrix}
				1 & 0 & 0 & 0 \\
				0 & 1 & 0 & 0 \\
				0 & 0 & \omega & 0 \\
				0 & 0 & 0 & \omega \\	
			\end{pmatrix} , \hspace{2ex}
		\eta_h= 
			\begin{pmatrix}
				1 & 1 & 0 & 0 \\
				-3 & 1 & 0 & 0 \\
				0 & 0 & 1 & 1 \\
				0 & 0 & -3 & 1 \\	
			\end{pmatrix},
	\end{equation}
	and $\omega$ is a primitive cubic root of unity. The fixed loci of $\eta_g$ and $\eta_h$ are 
	$\Fix(\eta_g) := \{ P \in \mathbb{P}^3 \mid \eta_g (P) = P \} = l_1 \cup l_2$, and
	$\Fix(\eta_h) := \{ P \in \mathbb{P}^3 \mid \eta_h (P) = P \} = l_3 \cup l_4$, where
	$l_1: X=Y=0$, $l_2:Z=W=0$, $l_3: -\sqrt{-3}X+Y= -\sqrt{-3}Z+W =0$, and $l_4: \sqrt{-3}X+Y= \sqrt{-3}Z+W =0$.
	Note that $\# (C \cap l_i ) =3$ ($i=1,2,3,4$) and $C \cap l_i \cap l_j = \emptyset$ if $i \ne j$. 
	
	\begin{claim} \label{claim:action-of-trace-0-involution-on-lines}
		For a $\rho \in G_l$ with trace $0$, where $l$ is a $K_4$-line, we have that $\rho(l_1)= l_2, \rho(l_2)= l_1$, $\rho(l_3)=l_4$, and $\rho(l_4)=l_3$. 
	\end{claim}
	\begin{claimproof}
		Since $g^1_3 \circ \rho \simeq g^1_3$, and the set of ramification points of $g^1_3$ is  $C \cap (l_1 \cup l_2)$, we have that $\rho(C \cap (l_1 \cup l_2)) = C \cap (l_1 \cup l_2)$. So,  $\rho(l_1 \cup l_2) = l_1 \cup l_2$. Similarly, $\rho(l_3 \cup l_4) = l_3 \cup l_4$. 
		Assume  $\rho(l_1) = l_1$. Then, $\rho(l_2) = l_2$. Since $\rho$ is an involution, there exist points $P \in C \cap l_1$ and $Q \in C \cap l_2$ such that $\rho(P) = P$ and $\rho(Q) = Q$. By Lemma~\ref{lemma:fixed-points-of-trace-0-involution}, $\Fix(\rho) \cap C = \{P ,Q\}$. Considering the decomposition of $\pi_l$ into $C  \rightarrow C/\langle \rho \rangle \rightarrow C/{G_l} \cong \mathbb{P}^1$, we see that $\pi_l(P) = \pi_l(Q)$.   Let $\tau$ and $\sigma$ be involutions with nonzero trace in $G_l$, that is $G_l = \{e, \sigma, \tau, \rho=\sigma\tau \}$. Since the Galois group $G_l$ acts on a fiber of $\pi_l$ transitively, we see that $\sigma(P) = Q$. By Lemma~\ref{lemma:actions-for-trigonal-maps}, $g^1_3 \circ \sigma \simeq h^1_3$ and $h^1_3 \circ \sigma  \simeq g^1_3$. Hence, $\sigma(l_1 \cup l_2) = l_3 \cup l_4$. We have $\sigma(P) = Q \in l_3 \cup l_4$. This contradicts that $Q \in l_2$ and $l_2 \cap (l_3 \cup l_4) = \emptyset$. We conclude that $\rho (l_1) = l_2$. Similarly, we have $\rho(l_2)= l_1$, $\rho(l_3)=l_4$, and $\rho(l_4)=l_3$. Now we conclude Claim~\ref{claim:action-of-trace-0-involution-on-lines}.
	\end{claimproof}

	Let $l$ be a $K_4$-line and let $\rho \in G_l$ be the involution of trace $0$. By $\rho (l_1) = l_2$ and $\rho (l_2)=l_1$, $\rho$ must be 
	\[ \rho=
		\begin{pmatrix}
		O & A \\
		B & O	
		\end{pmatrix},
	\]
	where $A$ and $B$ are some matrices of size $2 \times 2$ and $O$ is the zero matrix of size $2 \times 2$. Since $\rho V(Q) = V(Q)$, $B = \lambda A$ for some $\lambda$. Since $\rho(l_1 \cap l_3) = l_2 \cap l_4$ and $\rho(l_1 \cap l_4) = l_2 \cap l_3$  the $A$ must be 
	\[ A=
		\begin{pmatrix}
		-d & 1 \\
		3 & d	
		\end{pmatrix},
	\]
	where $d \in k$. Since $\rho(l_1 \cap C) = l_2 \cap C$, we have $d^3+c d^2 -9d -c =0$. Since $\rho^*(F) \in I(C)=(Q, F)$, 
	$\rho^*(F) = GQ + \mu F, $
	where $G$ is a homogeneous polynomial of degree $1$ and $\mu \in k$. From the computation of $\rho^*F$, we see that there is no nonzero term containing the factor $XW$. Hence $G$ must be zero. From $\rho^*(F) = \mu F$, we have $\lambda^3 = c^2 +27$. Namely, 
	\[ \rho=
		\begin{pmatrix}
			0 & 0 & -d & 1 \\
			0 & 0 & 3 & d \\
			-\lambda d & \lambda & 0 & 0 \\
			3\lambda & \lambda d & 0 & 0
		\end{pmatrix}
	\]
	with $d^3+c d^2 -9d -c =0$ and $\lambda^3 = c^2 +27$. The number of such $\rho$'s is at most $9$. By Proposition~\ref{prop:two-G_ls-belonging-to-K_4-lines}, we conclude that the number of $K_4$-lines is at most $9$.   
\end{proof}

%%%%%%%%%%%%%%%%%%%%%%%%%%%%%%%5
\section{Example attaining maximal numbers of $S_3$-lines and $K_4$-lines}
\label{section:examples}

In this section, we present an explicit example of a canonical curve of genus~$4$ that simultaneously attains the maximal numbers of $S_3$-lines and $K_4$-lines. This example demonstrates the sharpness of the bounds established in Theorem~\ref{theorem:main}.

\subsection*{The curve}
We consider the canonical curve $C \subset \mathbb{P}^3$ defined by
\begin{equation}\label{Eq:S5-curve}
	Q = XW + YZ = 0, \qquad
	F = X^2Z - Y^2X - Z^2W + W^2Y = 0.
\end{equation}
This curve appears in the classification of Kuribayashi--Kuribayashi \cite{KuribayashiKuribayashi1986, KuribayashiKuribayashi1990} as a canonical curve of genus~$4$ admitting the maximal possible automorphism group $\Aut(C) \cong S_5$.
\footnote{
	In \cite{KuribayashiKuribayashi1986}, $C$ is defined by $Q = XW + YZ = 0$ and $F = X^2Z - Y^2X - \underline{\mathbf{Z^2X}} + W^2Y = 0$. The cubic has a typographical error; the corrected cubic equation used here yields the canonical genus~$4$ curve with $\Aut(C) = G(120) \cong S_5$ described in their classification.
}
%In fact, $\Aut(C) \cong S_5$, a group of order~$120$, which is the maximal possible order of the automorphism group of a compact Riemann surface of genus~$4$.

\subsection*{Explicit Galois lines and Galois groups}

By explicit computer-assisted computation, we find that the curve $C$ admits exactly ten distinct $S_3$-lines and fifteen distinct $K_4$-lines. Hence, this example simultaneously attains the upper bounds for the numbers of $S_3$-lines and $K_4$-lines given in Theorem~\ref{theorem:main}.

Throughout, $\zeta$ denotes a fixed primitive fifth root of unity.
Let  
\[ 
	A:= \begin{pmatrix} 
			\zeta & 0 & 0 & 0 \\ 
			0 & \zeta^2 & 0 & 0 \\ 
			0 & 0 & \zeta^3 & 0 \\
	 		0 & 0 & 0 & \zeta^4 
	 	\end{pmatrix}, \qquad
	L:= \begin{pmatrix} 
	 		a & b & c & d \\ 
	 		b & d & a & c \\ 
	 		c & a & d & b \\ 
	 		d & c & b & a 
	 	\end{pmatrix}, \qquad
	D:= \begin{pmatrix} 
			0 & 0 & 1 & 0 \\ 
			1 & 0 & 0 & 0 \\ 
			0 & 0 & 0 & 1 \\ 
			0 & 1 & 0 & 0  
		\end{pmatrix},  
\] 
where  
\[ 
	a=-\frac{\sqrt5}{5}, \qquad b=-\frac12-\frac{\sqrt5}{10}, \qquad c=-\frac12+\frac{\sqrt5}{10}, \qquad d=\frac{\sqrt5}{5}. 
\] 
By \cite{KuribayashiKuribayashi1986,KuribayashiKuribayashi1990}, 
\[ 
	\Aut(C)=\langle A,L,D\rangle \cong S_5. 
\] 
The defining equations of the $S_3$-lines and $K_4$-lines, together with generators of their corresponding Galois groups expressed as words in $A$, $L$, and $D$, are listed in Tables~\ref{table:S3-lines-groups} and~\ref{table:K4-lines-groups}.

\begin{table}[H]
	\centering 
	\caption{The $S_3$-lines of $C$ and their Galois groups} 
	\label{table:S3-lines-groups} 
	\resizebox{\textwidth}{!}{ 
		\begin{tabular}{c|c|c} 
			\toprule 
			$S_3$-line $l$ & Defining equations of $l$ & Galois group $G_l$ \\
			\midrule 
			$l_1$ & \makecell[l]{ $(-\zeta-\zeta^2-\zeta^4)X+(\zeta+\zeta^3)Y+Z=0$\\ $(\zeta+\zeta^3+\zeta^4)X+\zeta^4Y+W=0$ } & $\langle LD^2A,\ ALADL\rangle$ \\ \addlinespace 
			$l_2$ & \makecell[l]{ $(-\zeta^2-\zeta^3-\zeta^4)X+(\zeta+\zeta^2)Y+Z=0$\\ $(\zeta+\zeta^2+\zeta^3)X+\zeta^3Y+W=0$ } & $\langle LA^2,\ ALA^2D\rangle$ \\ \addlinespace 
			$l_3$ & \makecell[l]{ $(-\zeta-\zeta^2-\zeta^3)X+(\zeta^3+\zeta^4)Y+Z=0$\\ $(\zeta^2+\zeta^3+\zeta^4)X+\zeta^2Y+W=0$ } & $\langle A^2L,\ ADLAL\rangle$ \\ \addlinespace
			$l_4$ & \makecell[l]{ $(-\zeta-\zeta^3-\zeta^4)X+(\zeta^2+\zeta^4)Y+Z=0$\\ $(\zeta+\zeta^2+\zeta^4)X+\zeta Y+W=0$ } & $\langle ALD^2,\ ADALA\rangle$ \\ \addlinespace
			$l_5$ & \makecell[l]{ $(\zeta^2+\zeta^3)X+(\zeta^2+\zeta^3)Y+Z=0$\\ $(-\zeta^2-\zeta^3)X+Y+W=0$ } & $\langle DL,\ ALA\rangle$ \\ \addlinespace 
			$l_6$ & \makecell[l]{ $(\zeta+\zeta^3)X+(-\zeta-\zeta^3-\zeta^4)Y+Z=0$\\ $(-\zeta^2-\zeta^4)X+\zeta^2Y+W=0$ } & $\langle DL,\ A(AL)^2\rangle$ \\ \addlinespace 
			$l_7$ & \makecell[l]{ $(\zeta^3+\zeta^4)X+(-\zeta^2-\zeta^3-\zeta^4)Y+Z=0$\\ $(-\zeta-\zeta^2)X+\zeta Y+W=0$ } & $\langle LD,\ A^2LA\rangle$ \\ \addlinespace 
			$l_8$ & \makecell[l]{ $(\zeta+\zeta^2)X+(-\zeta-\zeta^2-\zeta^3)Y+Z=0$\\ $(-\zeta^3-\zeta^4)X+\zeta^4Y+W=0$ } & $\langle LD,\ ALA^2\rangle$ \\ \addlinespace
			$l_9$ & \makecell[l]{ $(\zeta+\zeta^4)X+(\zeta+\zeta^4)Y+Z=0$\\ $(-\zeta-\zeta^4)X+Y+W=0$ } & $\langle LD,\ (AL)^2A\rangle$ \\ \addlinespace 
			$l_{10}$ & \makecell[l]{ $(\zeta^2+\zeta^4)X+(-\zeta-\zeta^2-\zeta^4)Y+Z=0$\\ $(-\zeta-\zeta^3)X+\zeta^3Y+W=0$ } & $\langle DL,\ ADLAL\rangle$ \\ 
			\bottomrule 
		\end{tabular}
	} 
\end{table}

\begin{table}[H]
	\centering 
	\caption{The $K_4$-lines of $C$ and their Galois groups} 
	\label{table:K4-lines-groups} 
	\resizebox{\textwidth}{!}{ 
		\begin{tabular}{c|c|c} 
			\toprule 
			$K_4$-line $l$ & Defining equations of $l$ & Galois group $G_l$ \\
			\midrule 
			$l'_1$ & \makecell[l]{ $Y+Z=0$\\ $X+W=0$ } & $\langle LD,\ DL\rangle$ \\ \addlinespace
			$l'_2$ & \makecell[l]{ $\zeta^2Y+Z=0$\\ $\zeta X+W=0$ } & $\langle D^2A,\ ALADL\rangle$ \\ \addlinespace 
			$l'_3$ & \makecell[l]{ $\zeta^4Y+Z=0$\\ $\zeta^2X+W=0$ } & $\langle DAD,\ ADALA\rangle$ \\ \addlinespace
			$l'_4$ & \makecell[l]{ $\zeta Y+Z=0$\\ $\zeta^3X+W=0$ } & $\langle A^2D^2,\ ALA^2D\rangle$ \\ \addlinespace
			$l'_5$ & \makecell[l]{ $\zeta^3Y+Z=0$\\ $\zeta^4X+W=0$ } & $\langle AD^2,\ DALA^2\rangle$ \\ \addlinespace
			$l'_6$ & \makecell[l]{ $(\zeta+\zeta^4)X+(\zeta^2+\zeta^3)Y+Z=0$\\ $(-\zeta^2-\zeta^3)X+ (\zeta+2\zeta^2+2\zeta^3+\zeta^4)Y+W=0$ } & $\langle L,\ DALA^2\rangle$ \\ \addlinespace 
			$l'_7$ & \makecell[l]{ $(\zeta^2+\zeta^3)X+(\zeta+\zeta^4)Y+Z=0$\\ $(-\zeta-\zeta^4)X+ (2\zeta+\zeta^2+\zeta^3+2\zeta^4)Y+W=0$ } & $\langle LD^2,\ ALA^2D\rangle$ \\ \addlinespace 
			$l'_8$ & \makecell[l]{ $(\zeta^2+\zeta^4)X+(\zeta+\zeta^2)Y+Z=0$\\ $(\zeta+\zeta^2+\zeta^3)X+ (-\zeta^2-2\zeta^3-\zeta^4)Y+W=0$ } & $\langle (LA)^2D,\ LD^2AL\rangle$ \\ \addlinespace 
			$l'_9$ & \makecell[l]{ $(-\zeta^2-\zeta^3-\zeta^4)X+(-\zeta-\zeta^2-\zeta^4)Y+Z=0$\\ $(-\zeta-\zeta^3)X+(\zeta^2-\zeta^3+\zeta^4)Y+W=0$ } & $\langle LD,\ D(LA)^2\rangle$ \\ \addlinespace 
			$l'_{10}$ & \makecell[l]{ $(-\zeta-\zeta^2-\zeta^3)X+(-\zeta-\zeta^3-\zeta^4)Y+Z=0$\\ $(-\zeta^2-\zeta^4)X+(\zeta-\zeta^2+\zeta^3)Y+W=0$ } & $\langle LD,\ ALADL\rangle$ \\ \addlinespace 
			$l'_{11}$ & \makecell[l]{ $(-\zeta-\zeta^3-\zeta^4)X+(-\zeta^2-\zeta^3-\zeta^4)Y+Z=0$\\ $(-\zeta-\zeta^2)X+(-\zeta+\zeta^3+\zeta^4)Y+W=0$ } & $\langle DL,\ (LA)^2D\rangle$ \\ \addlinespace 
			$l'_{12}$ & \makecell[l]{ $(-\zeta-\zeta^2-\zeta^4)X+(-\zeta-\zeta^2-\zeta^3)Y+Z=0$\\ $(-\zeta^3-\zeta^4)X+(\zeta+\zeta^2-\zeta^4)Y+W=0$ } & $\langle DL,\ ADALA\rangle$ \\ \addlinespace 
			$l'_{13}$ & \makecell[l]{ $(\zeta^3+\zeta^4)X+(\zeta^2+\zeta^4)Y+Z=0$\\ $(\zeta+\zeta^2+\zeta^4)X+(-2\zeta-\zeta^3-\zeta^4)Y+W=0$ } & $\langle ALADL,\ ADLAL\rangle$ \\ \addlinespace 
			$l'_{14}$ & \makecell[l]{ $(\zeta+\zeta^2)X+(\zeta+\zeta^3)Y+Z=0$\\ $(\zeta+\zeta^3+\zeta^4)X+(-\zeta-\zeta^2-2\zeta^4)Y+W=0$ } & $\langle ALA^2D,\ LDADL\rangle$ \\ \addlinespace 
			$l'_{15}$ & \makecell[l]{ $(\zeta+\zeta^3)X+(\zeta^3+\zeta^4)Y+Z=0$\\ $(\zeta^2+\zeta^3+\zeta^4)X+(-\zeta-2\zeta^2-\zeta^3)Y+W=0$ } & $\langle ADALA,\ LALD^2\rangle$ \\ 
			\bottomrule 
		\end{tabular}
	} 
\end{table}

Tables~\ref{table:S3-lines-groups} and~\ref{table:K4-lines-groups} were generated using GAP. The corresponding GAP program is reproduced in Appendix~\ref{appendix:gap}.

\subsection*{An illustrative example}

As an illustration, we explain how to verify that the line $l_1$ is an $S_3$-line. The projection from $l_1$ is given explicitly by
\[
	\pi_{l_1} : (X:Y:Z:W) \mapsto
	\bigl(
	(-\zeta-\zeta^2-\zeta^4)X+(\zeta+\zeta^3)Y+Z:
	(\zeta+\zeta^3+\zeta^4)X+\zeta^4Y+W
	\bigr).
\]
According to Table~\ref{table:S3-lines-groups}, let $\sigma_1$ and $\tau_1$ be the projective transformations represented by the matrices $LD^2A$ and $ALADL$, respectively. Then $G_{l_1}=\langle \sigma_1,\tau_1\rangle$. By direct computer calculation, one verifies that
\[
	\sigma_1, \tau_1 \in \Aut(C), \qquad
	\pi_{l_1} \circ \sigma_1 = \pi_{l_1}, \qquad
	\pi_{l_1} \circ \tau_1 = \pi_{l_1},
\]
and that $G_{l_1} \cong S_3$. Since every element of $G_{l_1}$ preserves $\pi_{l_1}$, $|G_{l_1}|$ divides $\deg\pi_{l_1}$. Since $|G_{l_1}|=6$ and $\deg\pi_{l_1} \le \deg C=6$, we obtain $|G_{l_1}|=\deg\pi_{l_1}=6$. Hence $\pi_{l_1}$ is a Galois covering with Galois group $G_{l_1}$. Consequently, the line $l_1$ is an $S_3$-line.

\subsection*{Method summary}
We briefly summarize the method used in \cite{Kato2024, KatoMaster2026} to determine all $S_3$-lines and $K_4$-lines on the curve~$C$. Throughout this section, we use the notation $G(120)$, $G(2 \times 3)$, $G(2 \times 3, 2 \times 6)$, $G(3 \times 2)$, $G(2 \times 2)$, and $G(2 + 2)$ introduced by Kuribayashi--Kuribayashi \cite{KuribayashiKuribayashi1986, KuribayashiKuribayashi1990} for certain finite subgroups of $GL(4,\mathbb{C})$.

\medskip\noindent
\textbf{Step 1: Identification of automorphisms.}
By the classification of Kuribayashi--Kuribayashi \cite{KuribayashiKuribayashi1986, KuribayashiKuribayashi1990}, the automorphism group of $C$ is $\Aut(C) = G(120) \cong S_5$. Moreover, this classification describes explicitly the induced action of $\Aut(C)$ on the space of holomorphic\/ $1$-forms. Via the canonical embedding, this yields a faithful linear representation $\Aut(C) \subset GL(4,\mathbb{C})$, which is the starting point of our computations.

\medskip\noindent
\textbf{Step 2: Enumeration of candidate subgroups.} 
Using the above representation, we enumerate all subgroups of\/ $\Aut(C)$
isomorphic to $S_3$ or $K_4$. According to \cite{KuribayashiKuribayashi1986, KuribayashiKuribayashi1990}, every $S_3$-subgroup of\/ $\Aut(C)$ is conjugate in $GL(4,\mathbb{C})$ to one of
\[
	G(2 \times 3), \quad G(2 \times 3, 2 \times 6), \quad G(3 \times 2),
\]
while every $K_4$-subgroup is conjugate in $GL(4,\mathbb{C})$ to either
\[
	G(2 \times 2) \quad \text{or} \quad G(2 + 2).
\]

\medskip\noindent
\textbf{Step 3: Selection of Galois subgroups.}
Not every such subgroup gives rise to a Galois line. By Proposition~\ref{prop:generator-of-Gal-S_3-line}, an $S_3$-subgroup corresponds to an $S_3$-line if and only if it is conjugate in $PGL(4,\mathbb{C})$ to the group defined by the representation \eqref{eq:representation-of-S_3}. Among the above possibilities, this condition is satisfied only for $G(2 \times 3)$.

Similarly, by Proposition~\ref{prop:generator-of-Gal-K4-line}, a $K_4$-subgroup corresponds to a $K_4$-line if and only if it is conjugate in $PGL(4,\mathbb{C})$ to the group defined by \eqref{eq:standard-representation-K4}, which occurs only for $G(2 \times 2)$.

In practice, the relevant conjugacy classes can be distinguished by comparing traces of group elements, which provides an effective and computationally inexpensive criterion. Indeed, for the candidate groups appearing in the classification of \cite{KuribayashiKuribayashi1986, KuribayashiKuribayashi1990} that arise in the $S_3$- and $K_4$-cases, the trace sets of the groups are pairwise distinct.

\medskip\noindent
\textbf{Step 4: Recovery of the Galois lines.}
Once a Galois subgroup $G_l \subset \Aut(C)$ has been identified, the corresponding Galois line is recovered from fixed-point data.

For an $S_3$-subgroup $G_l$, we choose an element $\sigma \in G_l$ of order~$3$ and compute its fixed locus $\Fix(\sigma)$. This locus is the union of a line and two points. By Propositions~\ref{prop:generator-of-Gal-S_3-line} and \ref{prop:distinct_l_induce_distinct_G_l}, the line joining these two points is precisely the associated $S_3$-line.

For a $K_4$-subgroup $G_l$, we choose two elements $\tau_1, \tau_2 \in G_l$ of order~$2$ with nonzero trace. Each fixed locus $\Fix(\tau_i)$ consists of a plane and a point $P_i$. By Remark~\ref{rem:obtain_l_from_G_l=K4}, the line joining the two points $P_1$ and $P_2$ is the corresponding $K_4$-line.

\medskip\noindent
\textbf{Implementation.}
All of the above steps reduce to finitely many group-theoretic and linear-algebraic computations. The enumeration of subgroups, conjugacy tests, and fixed locus calculations were carried out using GAP. The complete GAP program used in these computations is reproduced in Appendix~\ref{appendix:gap}.

\begin{remark}
	We make a remark on cyclic Galois lines. By applying the method summarized above in an analogous way, one can investigate cyclic Galois lines, namely $C_m$-lines. For the curve $C$ considered here, such an analysis shows that no cyclic Galois line exists.

	We also give an alternative argument. For this curve $C$, we have $\rank Q = 4$, so that $C$ admits exactly two trigonal pencils. Since the present curve admits fifteen $K_4$-lines, Proposition~\ref{prop:two_trigonal_are_Gal_K4Lines} implies that the two trigonal pencils cannot both be Galois.  
	On the other hand, since $C$ admits a $K_4$-line, Lemma~\ref{lemma:actions-for-trigonal-maps} implies that the two trigonal pencils are either both Galois or both non-Galois. 
	It follows that neither of the two trigonal pencils is Galois.
	Moreover, by \cite[Corollary~3.5]{KomedaTakahashi2021}, the existence of a $C_4$-line implies $\rank Q = 3$. Since $\rank Q = 4$ for the present curve, there is no $C_4$-line. Similarly, \cite[Corollary~3.8]{KomedaTakahashi2021} shows that there is no $C_5$-line, and \cite[Corollary~2.6]{KomedaTakahashi2022} shows that there is no $C_6$-line. Hence, the curve $C$ admits no cyclic Galois line.
\end{remark}

\appendix 
\section{GAP computation of Galois lines} 
\label{appendix:gap} 

The $S_3$-lines and $K_4$-lines listed in Tables~\ref{table:S3-lines-groups} and~\ref{table:K4-lines-groups} were obtained by the following GAP computation. Starting from the generators $A$, $L$, and $D$ of the group $\operatorname{Aut}(C) = G(120) \cong S_5$, the program enumerates all subgroups isomorphic to $S_3$ and $K_4$, selects those corresponding to Galois lines, recovers the defining equations of the lines, and computes shortest generating pairs for the corresponding Galois groups.  

The computation of shortest generating pairs is not required for the identification of Galois lines themselves. Its purpose is to obtain concise descriptions of the corresponding Galois groups as words in the generators $A$, $L$, and $D$, which are the generators displayed in Tables~\ref{table:S3-lines-groups} and~\ref{table:K4-lines-groups}. To this end, the program first computes a shortest word representative for each element of $\operatorname{Aut}(C)$ and then selects generating pairs whose total word length is minimal.

\begin{Verbatim}[fontsize=\footnotesize]
############################################################
# Step 1: Identification of automorphisms
############################################################

zeta := E(5);

A := DiagonalMat([zeta, zeta^2, zeta^3, zeta^4]);

a := -ER(5)/5;
b := -1/2 - ER(5)/10;
c := -1/2 + ER(5)/10;
d := ER(5)/5;
L := [[a,b,c,d],
      [b,d,a,c],
      [c,a,d,b],
      [d,c,b,a]];
      
D := [[0,0,1,0],
      [1,0,0,0],
      [0,0,0,1],
      [0,1,0,0]];

G := Group(A, L, D);
Print("|G(120)| = ", Size(G), "\n");

hom := IsomorphismPermGroup(G);
P := Image(hom);
Print("G(120) is isomorphic to S_5: ",
      IsomorphismGroups(P, SymmetricGroup(5)) <> fail, "\n");

epi := EpimorphismFromFreeGroup(G : names := ["A", "L", "D"]);

############################################################
# Step 2: Enumeration of candidate subgroups
############################################################

allSubs := Concatenation(List(ConjugacyClassesSubgroups(P), Elements));

candS3 := Filtered(allSubs, H -> Size(H) = 6 and not IsAbelian(H));
candK4 := Filtered(allSubs, H -> Size(H) = 4 and Exponent(H) = 2);

candS3 := List(candS3, H -> Group(List(GeneratorsOfGroup(H),
                            x -> PreImagesRepresentative(hom, x))));
candK4 := List(candK4, H -> Group(List(GeneratorsOfGroup(H),
                            x -> PreImagesRepresentative(hom, x))));

Print("subgroups of G(120): ", Length(allSubs),
      ";  candidates of type S_3: ", Length(candS3),
      ", of type K_4: ", Length(candK4), "\n");

############################################################
# Step 3: Selection of Galois subgroups
############################################################

refS3 := Group(DiagonalMat([E(3), E(3)^2, 1, 1]),
               [[0,1,0,0],[1,0,0,0],[0,0,-1,0],[0,0,0,-1]]);
refK4 := Group(DiagonalMat([1, 1, -1, -1]),
               DiagonalMat([-1, -1, -1, 1]));

tracesS3 := SortedList(List(Elements(refS3), TraceMat));
tracesK4 := SortedList(List(Elements(refK4), TraceMat));

galoisS3 := Filtered(candS3,
    H -> SortedList(List(Elements(H), TraceMat)) = tracesS3);
galoisK4 := Filtered(candK4,
    H -> SortedList(List(Elements(H), TraceMat)) = tracesK4);

Print("Galois subgroups selected: ", Length(galoisS3), " of type S_3, ",
      Length(galoisK4), " of type K_4\n");

############################################################
# Step 4: Recovery of Galois lines
############################################################

linesS3 := [];
for H in galoisS3 do
    sigma  := First(Elements(H), g -> Order(g) = 3);
    spaces := Filtered(Eigenspaces(CF(15), TransposedMat(sigma)),
                       sp -> Dimension(sp) = 1);
    points := List(spaces, sp -> BasisVectors(Basis(sp))[1]);
    forms  := NullspaceMat(TransposedMat(points));
    tri    := TriangulizedMat(List(forms, Reversed));
    Add(linesS3, rec(
        eqZ   := Reversed(tri[2]),
        eqW   := Reversed(tri[1]),
        group := H,
        words := List(SmallGeneratingSet(H),
                      g -> PreImagesRepresentative(epi, g))));
od;

linesK4 := [];
for H in galoisK4 do
    involutions := Filtered(Elements(H), g -> TraceMat(g) = -2);
    points := List(involutions, g -> BasisVectors(Basis(
        First(Eigenspaces(CF(15), TransposedMat(g)),
              sp -> Dimension(sp) = 1)))[1]);
    forms  := NullspaceMat(TransposedMat(points));
    tri    := TriangulizedMat(List(forms, Reversed));
    Add(linesK4, rec(
        eqZ   := Reversed(tri[2]),
        eqW   := Reversed(tri[1]),
        group := H,
        words := List(SmallGeneratingSet(H),
                      g -> PreImagesRepresentative(epi, g))));
od;

############################################################
# Step 5: Computation of shortest words and generating pairs
############################################################

gensG := [A, L, D];
F := FreeMonoid(["A", "L", "D"]);
gensF := GeneratorsOfMonoid(F);

elementsG := [One(G)];
wordsG := [One(F)];
frontier := [1];

while Length(elementsG) < Size(G) do
    next := [];
    for i in frontier do
        for k in [1..Length(gensG)] do
            g := elementsG[i] * gensG[k];
            if not g in elementsG then
                Add(elementsG, g);
                Add(wordsG, wordsG[i] * gensF[k]);
                Add(next, Length(elementsG));
            fi;
        od;
    od;
    frontier := next;
od;

Print("\n== minimal words ==\n");
Print("maximal length of a minimal word in G(120): ",
      Maximum(List(wordsG, Length)), "\n");

# shortest generating pairs

minimalPairsS3 := [];
for r in linesS3 do
    pairs := Filtered(Combinations(Elements(r.group), 2),
                      p -> Size(Group(p)) = Size(r.group));
    keys := List(pairs, 
                 p -> List(p, 
                           g -> wordsG[Position(elementsG, g)]));
    keys := List(keys, 
                 w -> [Length(w[1]) + Length(w[2]), 
                       SortedList(w)]);
    Add(minimalPairsS3, Minimum(keys));
od;

minimalPairsK4 := [];
for r in linesK4 do
    pairs := Filtered(Combinations(Elements(r.group), 2),
                      p -> Size(Group(p)) = Size(r.group));
    keys := List(pairs, 
                 p -> List(p, g -> wordsG[Position(elementsG, g)]));
    keys := List(keys, 
                 w -> [Length(w[1]) + Length(w[2]), SortedList(w)]);
    Add(minimalPairsK4, Minimum(keys));
od;

############################################################
# Step 6: Output of Galois lines and groups
############################################################

Print("\nz = E(5), a primitive fifth root of unity\n");
for t in [
        ["S_3", linesS3, minimalPairsS3, "l_"], 
        ["K_4", linesK4, minimalPairsK4, "l'_"]
    ] do
    Print("\n== ", t[1], "-lines ==\n");
    for i in [1..Length(t[2])] do
        eqStrings := [];
        for eq in [t[2][i].eqZ, t[2][i].eqW] do
            eqString := "";
            for j in [1..4] do
                if eq[j] = 0 then
                    continue;
                fi;
                negative := false;
                if eq[j] = 1 then
                    coeff := "";
                elif eq[j] = -1 then
                    coeff := "";
                    negative := true;
                else
                    coeff := Filtered(
                                 ReplacedString(
                                      String(eq[j]), "E(5)", "z"), 
                                 ch -> ch <> '*');
                    if PositionProperty(coeff{[2..Length(coeff)]},
                               ch -> ch = '+' or ch = '-') <> fail then
                        coeff := Concatenation("(", coeff, ")");
                    elif coeff[1] = '-' then
                        coeff := coeff{[2..Length(coeff)]};
                        negative := true;
                    fi;
                fi;
                if eqString = "" then
                    if negative then
                        Append(eqString, "-");
                    fi;
                elif negative then
                    Append(eqString, " - ");
                else
                    Append(eqString, " + ");
                fi;
                Append(eqString, coeff);
                Append(eqString, ["X","Y","Z","W"][j]);
            od;
            Append(eqString, " = 0");
            Add(eqStrings, eqString);
        od;
        Print(t[4], i, ":  ", eqStrings[1], "\n");
        Print(List(Concatenation(t[4], String(i), ":  "), 
              ch -> ' '), eqStrings[2], "\n");
        Print(
            List(
                Concatenation(t[4], String(i), ":  "), 
                ch -> ' '), 
            "G_l = < ",
            JoinStringsWithSeparator(
                List(
                    t[3][i][2],
                    w -> Filtered(String(w), ch -> ch <> '*')
                ), ", "), " >\n");
    od;
od;

Print("\nnumber of S_3-lines = ", Length(linesS3), 
      "  (distinct lines: ",
      Length(Set(List(linesS3, r -> [r.eqZ, r.eqW]))), ")\n");
Print("number of K_4-lines = ", Length(linesK4), 
      "  (distinct lines: ",
      Length(Set(List(linesK4, r -> [r.eqZ, r.eqW]))), ")\n");
\end{Verbatim}

%%%%%%%%%%%%%%%%%%%%%%%%%%%%%%%%%%%%%%%%%%%%%%%%%%%
\section*{Acknowledgments}
	The authors would like to thank Professor Akira Ohbuchi of Tokushima University for valuable discussions and helpful advice on the computational approach used in the work of the first author.

%%%%%%%%%%%%%%%%%%%%%%%%%%%%%%%%%%%%%%%%%%%%%%%%%%%

\bibliographystyle{amsplain}
\bibliography{references}

\end{document}